\documentclass[a4paper,12pt]{amsart}

\usepackage[utf8]{inputenc}
\usepackage[english]{babel}

\usepackage{enumerate}
\usepackage[usenames,dvipsnames,table]{xcolor}
\usepackage{color}
\usepackage{amsmath,amsthm,amssymb}
\usepackage{amsfonts}

\usepackage{fourier}

\usepackage{todonotes}
\usepackage{mathrsfs}

\usepackage[colorlinks=true,linkcolor=colorref,citecolor=colorcita,urlcolor=colorweb]{hyperref}
\definecolor{colorcita}{RGB}{21,86,130}
\definecolor{colorref}{RGB}{5,10,177}
\definecolor{colorweb}{RGB}{177,6,38}

\newcommand{\N}{\mathbb{N}}
\newcommand{\Z}{\mathbb{Z}}
\newcommand{\C}{\mathbb{C}}
\newcommand{\R}{\mathbb{R}}
\newcommand{\K}{\mathbb{K}}
\newcommand{\D}{\mathbb{D}}

\newtheorem{theo}{Theorem}[section]

\newtheorem{coro}[theo]{Corollary}
\newtheorem{prop}[theo]{Proposition}

\theoremstyle{definition}
\newtheorem{remark}[theo]{Remark}

\allowdisplaybreaks

\title{Mean ergodic weighted shifts on Köthe echelon spaces}

\author[T.\ Kalmes]{Thomas Kalmes}
\address{Thomas Kalmes, Chemnitz University of Technology, Faculty of Mathematics, 09107 Chemnitz, Germany}
\email{thomas.kalmes@math.tu-chemnitz.de}

\author[D.\ Santacreu]{Daniel Santacreu}
\address{Daniel Santacreu, Instituto Universitario Matem\'atica Pura y Aplicada IUMPA, Universitat Polit\`ecnica de Val\`encia, Camino de Vera, s/n, 46701 Valencia, Spain}
\email{dasanfe5@posgrado.upv.es}

\date{\today}

\begin{document}

\begin{abstract}
	Necessary and sufficient conditions are given for mean ergodicity, power boundedness, and topologizability for weighted backward shift and weighted forward shift operators, respectively, on K\"othe echelon spaces in terms of the weight sequence and the K\"othe matrix. These conditions are evaluated for the special case of power series spaces which allow for a characterization of said properties in many cases. In order to demonstrate the applicability of our conditions, we study the above properties for several classical operators on certain function spaces.
	
	\mbox{}\\
	
	\\ 
	
	
\end{abstract}

\maketitle

\section{Introduction}

The aim of this note is to investigate mean ergodicity and related properties of weighted shift operators on K\"othe echelon spaces. Recall that a continuous linear operator $T$ on a locally convex Hausdorff space (briefly, lcHs) $E$ is said to me \textit{mean ergodic} if the limits
$$\lim_{n\rightarrow\infty}\frac{1}{n}\sum_{m=1}^n T^m x, x\in E,$$
exist in $E$. Since the seminal result of von Neumann (1931) who proved that unitary operators on a Hilbert space are mean ergodic, numerous contributions have been made to the topic of mean ergodicity and its applications. For the special case when $E$ is a Banach space, the theory is particularly well-developed and rich, see e.g.\ \cite[Chapter VIII]{DuSc1958}, \cite{EiFaHaNa15}, and \cite{Krengel}, and references therein.

In recent years, several authors studied mean ergodicity and related properties of continuous linear operators on lcHs which are not Banach spaces (but mainly Fr\'echet spaces) both from an abstract point of view (see e.g.\ \cite{ABR2009,ABR2010, AlBoRi13, AlBoRi14}, \cite{BoPaRi11}, \cite{ Piszczek10b, Piszczek10}) as well as for concrete types of continuous linear operators (see e.g.\ \cite{AlJoMe22}, \cite{Be14}, \cite{BeBoFe13}, \cite{Beltran20},  \cite{BeGCJoJo16, BGJJ2016mw}, \cite{BeJo21}, \cite{BoDo2011A, BoDo11B},  \cite{BoJoRo18}, \cite{BoRi09}, \cite{FeGaJo18}, \cite{GCJoJo16}, \cite{HaShZh19}, \cite{JoSaSP20, JoSaSP21}, \cite{K2019p, K2020}, \cite{KaSa22}, \cite{Rodriguez19},  \cite{SeMeBo20}).

In the present article we study mean ergodicity, power boundedness and topologizability for weighted shift operators on K\"othe echelon spaces. While weighted shift operators on sequence spaces are a natural testing field to study operator theoretic properties, several operators on Fr\'echet spaces which occur naturally in analysis are conjugate to a weighted shift operator on a suitable K\"othe echelon space. Since the aforementioned properties are stable under conjugacy, our results will be used to investigate these properties for the Volterra operator and the differentiation operator on spaces of holomorphic functions as well as for the annihiliation and creation operator on the space of rapidly decreasing smooth functions.

The paper is organized as follows. In section \ref{sec: notation etc} we recall some notation and notions and we provide some basic results which will be used throughout. In section \ref{sec: Koethe spaces}, we give necessary and sufficient conditions for topologizability, power boundedness, and mean ergodicity of weighted shift operators on K\"othe echelon spaces $\lambda_p(A)$ in terms of the weight sequence and the K\"othe matrix $A$. In section \ref{sec: power series spaces} we evaluate our results for the special case of power series spaces. In this special situation of K\"othe echelon spaces, the special structure of the K\"othe matrix allows for particular simple characterizations of said properties in many cases. Concrete examples are discussed in order to demonstrate the applicability of our results.

\section{Notation and preliminary results}\label{sec: notation etc}
Throughout this article we use standard notation from functional analysis; we refer to \cite{Jarchow1981}, \cite{MeVo1997}. As usual, we denote by $\omega=\K^{\N_0}$ the vector space of all $\K$-valued sequences (where as usual $\K\in\{\R,\C\}$) equipped with the Fr\'echet space topology of coordinatewise convergence. For a fixed sequence $w=(w_n)_{n\in\N_0}\in\omega$ we define the corresponding weighted backward shift and weighted forward shift, respectively, as
\[B_w:\omega\rightarrow\omega, (x_n)_{n\in\N_0}\mapsto \left(w_n x_{n+1}\right)_{n\in\N_0}\]
and
\[F_w:\omega\rightarrow\omega, (x_n)_{n\in\N_0}\mapsto \left(w_n x_{n-1}\right)_{n\in\N_0},\]
where we use the notational convention $x_{-n}:=0$ for all $n\in\N$ which will be employed throughout this paper. Additionally, throughout, we use the convention $\frac{0}{0}:=0$, $\frac{\alpha}{0}:=\infty$ for $\alpha>0$ as well as the notation $e_r=\left(\delta_{m,r}\right)_{m\in\N_0}$ (Kronecker's $\delta$) for $r\in\N_0$. For the special case of $w_n=1$, $n\in\N_0$, we simply write $B$ and $F$ instead of $B_w$ and $F_w$, respectively.

Recall that a matrix of non-negative real numbers $A=(a_{n,k})_{n,k\in\N_0}$ is a \emph{Köthe matrix} if $0\leq a_{n,k}\leq a_{n,k+1}$ for every $n,k\in\N_0$, and if for each $n\in\N_0$ there is $k\in\N_0$ such that $a_{n,k}>0$. For $1\leq p<\infty$ we define as usual the \emph{Köthe echelon space} (\emph{of order} $p$) by
\[ \lambda_p(A):=\left\{ x\in \omega: \Vert x\Vert_{k,p}:=\Vert x\Vert_k:=\left( \sum_{n=0}^\infty |x_{n}a_{n,k}|^p\right)^{1/p}<\infty, \text{ for each } k\in\N_0 \right\}. \]
Analogously, for $p=\infty$, we have
\[ \lambda_\infty(A):=\left\{ x\in \omega: \Vert x\Vert_{k,\infty}:=\Vert x\Vert_k:=\sup_{n\in\N_0} |x_{n}a_{n,k}|<\infty, \text{ for each } k\in\N_0 \right\} \]
and
\[ \lambda_0(A):=\left\{ x\in \lambda_\infty(A): \lim_{n\rightarrow\infty}x_n a_{n,k}=0 \text{ for each } k\in\N_0 \right\}.\]
Then, for $1\leq p\leq \infty$,  $\lambda_p(A)$ is a Fr\'echet space with fundamental sequence of seminorms $(\Vert\cdot\Vert_{k,p})_{k\in\N_0}$, $\lambda_0(A)$ is a closed subspace of $\lambda_\infty(A)$, and $\lambda_p(A)$ is separable for $p\in[1,\infty)\cup\{0\}$.

Given an \emph{exponent sequence}, i.e.\ a monotonically increasing sequence $\alpha=(\alpha_n)_{n\in\N_0}$ in $[0,\infty)$ with $\lim_{n\rightarrow\infty}\alpha_n=\infty$. For arbitrary strictly increasing sequences $(s_k)_{k\in\N_0}$ with $\lim_{k\rightarrow\infty}s_k=0$ and $(t_k)_{k\in\N_0}$ with $\lim_{k\rightarrow\infty}t_k=\infty$ we define the K\"othe matrices $A_0(\alpha):=\left(\exp(s_k\alpha_n)\right)_{k,n\in\N_0}$ and $A_\infty(\alpha):=\left(\exp(t_k\alpha_n)\right)_{k,n\in\N_0}$ as well as $\Lambda_0(\alpha):=\lambda_1(A_0(\alpha))$ and $\Lambda_\infty(\alpha):=\lambda_1(A_\infty(\alpha))$. It is not hard to see that the definition of $\Lambda_0(\alpha)$ and $\Lambda_\infty(\alpha)$ does not depend on the particular choice of the sequences $(s_k)_{k\in\N_0}$ and $(t_k)_{k\in\N_0}$, respectively. $\Lambda_0(\alpha)$ is called \emph{power series space of finite type} and $\Lambda_\infty(\alpha)$ \emph{power series space of infinite type} (associated to $\alpha$).

As we are interested in ergodicity and related properties of the weighted backward and forward shift operators on $\lambda_p(A)$, we first have to characterize when $B_w$ and $F_w$ operate on these spaces.

\begin{prop}\label{prop: continuity}
	For a K\"othe matrix $A$, $p\in [1,\infty]\cup\{0\}$ and $w=(w_n)_{n\in\N_0}\in\omega$ the following results hold.
	\begin{itemize}
		\item[(a)] For $B_w$ the following are equivalent.
		\begin{itemize}
			\item[(i)] $B_w:\lambda_p(A)\rightarrow\lambda_p(A)$ is correctly defined.
			\item[(ii)] $B_w:\lambda_p(A)\rightarrow\lambda_p(A)$ is continuous.
			\item[(iii)] For every $k\in\N_0$ there are $l\in\N_0$ and $C>0$ such that
			$$\forall\,n\in\N_0:\,|w_n|a_{n,k}\leq C a_{n+1,l}.$$
		\end{itemize}
		\item[(b)] For $F_w$ the following are equivalent.
		\begin{itemize}
			\item[(i)] $F_w:\lambda_p(A)\rightarrow\lambda_p(A)$ is correctly defined.
			\item[(ii)] $F_w:\lambda_p(A)\rightarrow\lambda_p(A)$ is continuous.
			\item[(iii)] For every $k\in\N_0$ there are $l\in\N_0$ and $C>0$ such that
			$$\forall\,n\in\N:\,|w_n|a_{n,k}\leq C a_{n-1,l}.$$
		\end{itemize}
	\end{itemize}
\end{prop}

\begin{proof}
	We only prove part (a), the proof of part (b) is, mutatis mutandis, the same. Since $\lambda_p(A)$ is a Fr\'echet space and $B_w:\omega\rightarrow\omega$ is obviously continuous, a standard application of the Closed Graph Theorem shows that (i) implies (ii).
	
	Next, if (ii) holds, we conclude
	$$\forall\,k\in\N_0\,\exists\,l\in\N_0, C>0\,\forall\,x\in\lambda_p(A):\,\Vert B_w x\Vert_k\leq C\Vert x\Vert_l.$$
	Evaluating this inequality for $x=e_n$ shows that (iii) is true.
	
	Finally, (iii) trivially implies (i).
\end{proof}

The previous proposition yields the following result. Throughout the article we use the convention $\ln(0)=-\infty$.

\begin{coro}\label{cor: continuity on power series spaces}
	Let $\alpha=\left(\alpha_n\right)_{n\in\N_0}$ be an exponent sequence such that $\limsup_{n\rightarrow\infty}\alpha_{n+1}\alpha_n^{-1}<\infty$. Then, for $w\in\omega$ the following results hold.
	\begin{itemize}
		\item[(a)] For the weighted shifts $B_w$ and $F_w$ the following are equivalent.
		\begin{itemize}
			\item[(i)] $B_w$ is correctly defined on $\Lambda_\infty(\alpha)$.
			\item[(ii)] $B_w$ is continuous on $\Lambda_\infty(\alpha)$.
			\item[(iii)] $F_w$ is correctly defined on $\Lambda_\infty(\alpha)$.
			\item[(iv)] $F_w$ is continuous on $\Lambda_\infty(\alpha)$.
			\item[(v)]
			$\limsup\limits_{n\rightarrow\infty}\frac{\ln|w_n|}{\alpha_n}<\infty.$
		\end{itemize}
		\item[(b)] For the weighted shifts $B_w$ and $F_w$ the following are equivalent.
		\begin{itemize}
			\item[(i)] $B_w$ is correctly defined on $\Lambda_0(\alpha)$.
			\item[(ii)] $B_w$ is continuous on $\Lambda_0(\alpha)$.
			\item[(iii)] $F_w$ is correctly defined on $\Lambda_0(\alpha)$.
			\item[(iv)] $F_w$ is continuous on $\Lambda_0(\alpha)$.
			\item[(v)]
			$\limsup\limits_{n\rightarrow\infty}\frac{\ln|w_n|}{\alpha_n}\leq 0.$
		\end{itemize}
	\end{itemize}
\end{coro}

\begin{proof}
	By Proposition \ref{prop: continuity}, (i) and (ii) are equivalent, as are (iii) and (iv), in (a) and (b). We will show that (a) (v) is equivalent to condition (a) (iii) from Proposition \ref{prop: continuity} for the particular case of $\Lambda_\infty(\alpha)$ as well as to condition (b) (iii) from Proposition \ref{prop: continuity} which will finish the proof of part (a). Analogously, (b) will be proved once we have shown that (b) (v) is equivalent to condition (a) (iii) from Proposition \ref{prop: continuity} for the particular case of $\Lambda_0(\alpha)$ as well as to condition (b) (iii) from Proposition \ref{prop: continuity}. We set $M:=\limsup_{n\rightarrow\infty}\alpha_{n+1}\alpha_n^{-1}.$
	
	Let us show that (a) (ii) implies (a) (v). It follows from Proposition \ref{prop: continuity} that for each $k\in\N_0$ there are $l\in\N_0$ and $C>0$ with
	$$\forall\,n\in\N_0:\,|w_n|\leq C\exp\left(l\alpha_{n+1}-k\alpha_n\right).$$
	Let $M'>M$ be arbitrary and $N_0\in\N$ such that $\alpha_{n+1}\alpha_n^{-1}\leq M'$ for $n\geq N_0$. We conclude for $n\geq N_0$
	$$|w_n|\leq C\exp\left(\alpha_n(l M'-k)\right)$$
	which implies that $\limsup_{n\rightarrow\infty}\ln|w_n|(\alpha_n)^{-1}<\infty$, i.e.\ (a) (v) holds. 
	
	On the other hand, if (a) (v) holds, for fixed $k\in\N$ let $N_0\in\N$ and $c>0$ be such that 
	$$\forall\,n\geq N_0: \ln|w_n|\leq c k\alpha_n\mbox{ and } \alpha_n>0.$$
	Then, for $n\geq N_0$ we conclude
	$$|w_n|\exp(k\alpha_n)\leq \exp(k(c+1)\alpha_n)\leq\exp(k(c+1)\alpha_{n+1})$$
	so that condition (a) (iii) from Proposition \ref{prop: continuity} holds true for every $l\in\N$ with $l\geq k(c+1)$ which implies (a) (ii). 
	
	Next we assume that (a) (iv) holds. From Proposition \ref{prop: continuity} (b) we deduce for $k\in\N_0$ for suitable $l\in\N_0, C>0$
	$$\forall\,n\in\N:\,|w_n|\leq C\exp\left(l\alpha_{n-1}-k\alpha_n\right)\leq C\exp\left((l-k)\alpha_n\right).$$
	Thus, (a) (v) follows. On the other hand, if (a) (v) holds, for arbitrary $k\in\N$ there are $c>0$ and $N_0\in\N$ such that for every $n\geq N_0$ there hold $\alpha_n\leq M'\alpha_{n-1}$ as well as $\ln|w_n|\leq c k\alpha_n$. Hence, for $n\geq N_0$ we deduce $|w_n|\exp\left(k\alpha_n\right)\leq \exp\left((c+M')k\alpha_{n-1}\right)$ so that the condition (b) (iii) from Proposition \ref{prop: continuity} follows and $F_w$ is continuous on $\Lambda_\infty(\alpha)$. The proof of (a) is complete.
	
	In order to complete the proof of part (b), we first assume that $B_w$ is continuous on $\Lambda_0(\alpha)$. Thus, by Propositions \ref{prop: continuity} (a), for every $k\in\N$ there are $l\in\N$ and $C>0$ such that 
	$$\forall\,n\in\N_0:\,|w_n|\exp\left(-\frac{1}{k}\alpha_n\right)\leq C\exp\left(-\frac{1}{l}\alpha_{n+1}\right),$$
	so that
	$$\forall\,n\in\N_0: |w_n|\leq C\exp\left(\frac{\alpha_n}{k}-\frac{\alpha_{n+1}}{l}\right)\leq C\exp\left(\alpha_n\left(\frac{1}{k}-\frac{1}{l}\right)\right).$$
	Since $\alpha$ tends to infinity, there is $N_k$ such that $\ln(C)/\alpha_n\leq 1/k$ whenever $n\geq N_k$. Thus, by the above inequality we deduce $\sup_{n\geq N_k}\frac{\ln|w_n|}{\alpha_n}\leq \frac{2}{k}-\frac{1}{l}<\frac{2}{k}$. Since $k\in\N$ was chosen arbitrarily (b) (v) follows.
	
	On the other hand, if (b) (v) holds, we fix $k\in\N_0$ as well as $M'>M$ and we choose $l\in\N$ with $l/M'>(k+1)$. Let $N_0\in\N$ be such that  $\frac{1}{k+1}-\frac{M'}{l}>\frac{\ln|w_n|}{\alpha_n}$ as well as $\alpha_{n+1}\alpha^{-1}_n\leq M'$ for each $n\geq N_0$. For $n\geq N_0$ we conclude
	$$|w_n|\exp\left(-\frac{1}{k+1}\alpha_n\right)\leq\exp\left(-\frac{M'}{l}\alpha_n\right)\leq \exp\left(-\frac{1}{l}\alpha_{n+1}\right).$$
	Thus, conditions (a) (iii) from Proposition \ref{prop: continuity} holds which implies the continuity of $B_w$ on $\Lambda_0(\alpha)$.
	
	Finally, to finish the proof of part (b), we assume that $F_w$ is continuous on $\Lambda_0(\alpha)$. Thus, by Proposition \ref{prop: continuity} (b), we deduce for $k\in\N$ the existence of $C>0$ and $l\in\N$ such that
	$$\forall\,n\geq N_k:\,|w_n|\leq C\exp\left(\frac{1}{k}\alpha_n-\frac{1}{l}\alpha_{n-1}\right)\leq C\exp\left(\alpha_n\left(\frac{1}{k}-\frac{k}{M'l}\right)\right),$$
	where $N_k$ is such that $\alpha_n\leq M'\alpha_{n-1}$ as well as $c/\alpha_n\leq 1/k$ for every $n\geq N_k$. From the previous inequality we derive $\sup_{n\geq N_k}\ln|w_n|/\alpha_n\leq 2/k$. Because $k\in\N$ was chosen arbitrarily, the validity of (b) (v) follows. On the other hand, if (b) (v) is valid, we fix $k\in\N$. Moreover, let $N_k\in\N$ be such that $|w_n|\leq\exp\left(\alpha_n (2k)^{-1}\right)$ for each $n\geq N_k$. For every $n\geq N_k$
	$$|w_n|\exp\left(-\frac{1}{k}\alpha_n\right)\leq\exp\left(-\frac{1}{2k}\alpha_n\right).$$
	Hence, we deduce that condition (b) (iii) from Proposition \ref{prop: continuity} holds so that $F_w$ is continuous on $\Lambda_0(\alpha)$.
\end{proof}

\begin{remark}
	It should be noted that without the hypothesis $\limsup_{n\rightarrow\infty}\alpha_{n+1}\alpha_n^{-1}<\infty$ in Corollary \ref{cor: continuity on power series spaces}, condition (a)(v) implies (a)(ii) and follows from (a)(iv) while condition (b)(v) implies (b)(iv) and is implied by (b)(ii). 
\end{remark}

In the rest of this section we recall some notions and abstract results in order to motivate our considerations in the following section. Let $E$ be a locally convex Hausdorff space (briefly, lcHs) and $T\in \mathcal{L}(E)$, where as usual we denote by $\mathcal{L}(E)$ the space of continuous linear operators on $E$. $T$ is said to be {\it topologizable} if for every continuous seminorm $p$ on $E$ there is a continuous seminorm $q$ on $E$ such that for every $m\in\N$ there is $\gamma_m>0$ with
\[p\left(T^mx\right)\leq \gamma_m q(x) \text{ for all } x\in E.\]
For the special case that in the above inequality one can take $\gamma_m=1$ for all $m\in\N$ we say that $T$ is {\it power bounded}. In this case the family $\left\{T^m:\,m\in\N\right\}$ is an equicontinuous subset of $\mathcal{L}(E)$. Moreover, $T$ is {\it Cesàro bounded} if the family $\left\{T^{[n]}:\,n\in\N\right\}$ is an equicontinuous subset of $\mathcal{L}(E)$, where $T^{[n]}$ denotes the $n$-th Cesàro mean given by
\begin{equation*}
\frac{1}{n}\sum_{m=1}^{n} T^m.
\end{equation*}
An operator $T\in\mathcal{L}(E)$ is called {\it mean ergodic} if there is $P\in\mathcal{L}(E)$ such that $\lim_{n\to\infty}T^{[n]}x=Px$ for each $x\in E$. In case that the convergence is uniform on bounded subsets of $E$ then $T$ is called {\it uniformly mean ergodic}.

Let $F$ be a lcHs. An operator $S\in \mathscr{L}(F)$ is called {\it conjugate} to the operator $T\in\mathscr{L}(E)$ if there is a bijective, continous linear operator $\Phi:E\to F$ with continuous inverse such that $\Phi\circ T=S\circ \Phi$. It is not hard to see that all of the above properties of $T$ are stable under conjugacy, i.e.\ $S$ has any of the above properties if (and only if) $T$ does.

Clearly, every power bounded operator $T$ is Cesàro bounded and $(\frac{1}{n}T^n x)_{n\in\N}$ converges to $0$ for each $x\in E$. Moreover, on a barrelled space $E$ an operator $T$ is mean ergodic if and only if $(T^{[n]}x)_{n\in\N}$ converges for every $x\in E$. Additionally, on a barrelled space $E$, mean ergodic operators $T$ are Cesàro bounded and due to $\frac{1}{n}T^nx=T^{[n]}x-\frac{n-1}{n}T^{[n-1]}x$, the sequences $\left(\frac{1}{n}T^nx\right)_{n\in\N}, x\in E$, converge to zero. Conversely, as shown in \cite[Corollary 2.5]{ABR2009} (see also \cite[Theorem 2.3]{KaSa22}) the following useful result holds. Recall that a reflexive lcHs are always barrelled. 

\begin{theo}\label{theo: mean ergodic}
Let $E$ be a reflexive lcHs and $T\in \mathcal{L}(E)$. Then $T$ is mean ergodic if and only if $T$ is Cesàro bounded and $\lim_{n\to\infty} \frac{1}{n}T^n x=0$ for every $x\in E$.
\end{theo}

\begin{theo}\label{theo:ergodicity and transposed}
	{\rm (see \cite[Theorem 2.5]{KaSa22}).} Let $E$ be a Montel space and let $T\in \mathcal{L}(E)$.
	\begin{enumerate}
		\item[(a)] $T$ is mean ergodic if and only if $T$ is uniformly mean ergodic.
		\item[(b)] The following are equivalent.
		\begin{enumerate}
			\item[(i)] $T$ is Ces\`aro bounded and $\lim_{n\rightarrow\infty}\frac{T^n}{n}=0$, pointwise in $E$.
			\item[(ii)] $T$ is mean ergodic on $E$.
			\item[(iii)] $T$ is uniformly mean ergodic on $E$.
			\item[(iv)] $T^t$ is mean ergodic on $(E',\beta(E',E))$.
			\item[(v)] $T^t$ is uniformly mean ergodic on $(E',\beta(E',E))$.
			\item[(vi)] $T^t$ is Ces\`aro bounded on $(E',\beta(E',E))$ and $\lim_{n\rightarrow\infty}\frac{(T^{t})^n}{n}=0$, pointwise in $(E',\beta(E',E))$.
		\end{enumerate}
		Here, as usual, $\beta(E',E)$ denotes the strong dual topology on $E'$.
	\end{enumerate}
\end{theo}

In particular, power bounded operators on Montel spaces are uniformly mean ergodic (see \cite[p. 917]{BoPaRi11}). In addition, as for a mean ergodic operator $T$ on a lcHs $E$, for each $x\in E$ the sequence $\left(\frac{1}{n}T^nx\right)_{n\in\N}$ is in particular bounded, we see that under the additional hypothesis that $E$ is barrelled the family of operators $\left\{\frac{1}{n}T^n;\,n\in\N\right\}$ is equicontinuous. It follows that mean ergodic operators on barrelled spaces are topologizable.

\section{Ergodicity and related properties for weighted shifts}\label{sec: Koethe spaces}

We begin by characterizing when weighted backward shifts and weighted forward shifts are topologizable or power bounded on $\lambda_p(A)$. Setting $B_w^0=I$ and $F_w^0=I$, observe that for $m\in\N_0$ the $m$-th iterate of $B_w$ applied to $x=\left(x_n\right)_{n\in\N_0}$ yields
\[ B_w^mx=\left(x_{n+m}\prod_{j=0}^{m-1}w_{n+j}\right)_{n\in\N_0}=\left(x_{0+m}\prod_{j=0}^{m-1}w_{0+j} ,x_{1+m}\prod_{j=0}^{m-1}w_{1+j},x_{2+m}\prod_{j=0}^{m-1}w_{2+j},\dots\right)\]
while the $m$-th iterate of $F_w$ is given by
\[F_w^mx=\left(x_{n-m}\prod_{j=0}^{m-1}w_{n-j}\right)_{n\in\N_0}=\left(\underbrace{0,\ldots,0}_{m\textrm{-times}},x_0\prod_{j=1}^{m}w_{0+j},x_1\prod_{j=1}^{m}w_{1+j},x_2\prod_{j=1}^{m}w_{2+j},\ldots\right).\]
Thus, in case of $p\in[1,\infty)$, for the $k$-th seminorm we have 
\begin{equation}\label{seminorm Bm, p finite}
 \Vert B^m x\Vert_{k,p}^p= \sum_{n=0}^\infty \left|x_{n+m}\left(\prod_{j=0}^{m-1}w_{n+j}\right)a_{n,k}\right|^p=\sum_{n=m}^\infty \left|x_{n}\left(\prod_{j=1}^{m}w_{n-j}\right)a_{n-m,k}\right|^p
\end{equation}
and
\begin{equation}\label{seminorm Fm, p finite}
	\Vert F^m x\Vert_{k,p}^p= \sum_{n=0}^\infty \left|x_{n-m}\left(\prod_{j=0}^{m-1}w_{n-j}\right)a_{n,k}\right|^p=\sum_{n=0}^\infty \left|x_{n}\left(\prod_{j=1}^{m}w_{n+j}\right)a_{n+m,k}\right|^p
\end{equation}
while for $p\in\{0,\infty\}$ it follows
\begin{equation}\label{seminorm Bm, p infty}
	\Vert B^m x\Vert_{k,\infty}=\sup_{n\in\N_0}\left|x_{n+m}\left(\prod_{j=0}^{m-1}w_{n+j}\right)a_{n,k}\right|=\sup_{n\geq m}\left|x_n\left(\prod_{j=1}^mw_{n-j}\right)a_{n-m,k}\right|
\end{equation}
and
\begin{equation}\label{seminorm Fm, p infty}
	\Vert F^m x\Vert_{k,\infty}=\sup_{n\in\N_0}\left|x_{n-m}\left(\prod_{j=0}^{m-1}w_{n-j}\right)a_{n,k}\right|=\sup_{n\in\N_0}\left|x_n\left(\prod_{j=1}^mw_{n+j}\right)a_{n+m,k}\right|.
\end{equation}

\begin{prop}\label{prop: topologizable}
	For a K\"othe matrix $A$, $p\in [1,\infty]\cup\{0\}$, and $w\in\omega$ the following hold.
	\begin{itemize}
		\item[(a)] If $B_w$ is continuous on $\lambda_p(A)$ the following are equivalent.
		\begin{itemize}
			\item[(i)] $B_w$ is topologizable on $\lambda_p(A)$.
			\item[(ii)] For every $k\in\N_0$ there is $l\in \N_0$ such that for each $m\in\N_0$
			\begin{equation}\label{eq: characterization topologizability B}
				\sup_{n\in\N_0}\frac{\left|\prod_{j=0}^{m-1} w_{n+j}\right|a_{n,k}}{a_{n+m,l}}<\infty.
			\end{equation}
		\end{itemize} 
		\item[(b)] If $F_w$ is continuous on $\lambda_p(A)$ the following are equivalent.
		\begin{itemize}
			\item[(i)] $F_w$ is topologizable on $\lambda_p(A)$.
			\item[(ii)] For every $k\in\N_0$ there is $l\in \N_0$ such that for each $m\in\N_0$
			\begin{equation}\label{eq: characterization topologizability F}
				\sup_{n\in\N_0}\frac{\left|\prod_{j=1}^{m} w_{n+j}\right|a_{n+m,k}}{a_{n,l}}<\infty.
			\end{equation}
		\end{itemize} 
	\end{itemize}
\end{prop}

\begin{proof}
We only give the proof of part (a) since the proof of part (b) is along the exact same lines. Thus, let $B_w$ be continuous on $\lambda_p(A)$. Assume $B_w$ is topologizable. Given $k\in\N_0$ there is $l\in\N_0$ such that for every $m\in\N_0$ there is $c_m>0$ with
\[ \Vert B_w^m x\Vert_k\leq c_m \Vert x\Vert_l, \]
for every $x\in\lambda_p(A)$. By \eqref{seminorm Bm, p finite} and \eqref{seminorm Bm, p infty}, respectively, taking $x=e_{n+m}$ for $n,m\in\N_0$ we obtain 
\[\Vert B^m e_{n+m}\Vert_k= \left|\prod_{j=1}^{m}w_{n+m-j}\right| a_{n,k} \leq c_m a_{n+m,l}. \]
Thus, (ii) follows.

Conversely, if (ii) is valid, fix $k\in \N_0$. Let $l\in\N_0$ be such that for every $m\in\N_0$ there is $c_m>0$ with
\[ \frac{\left|\prod_{j=0}^{m-1} w_{n+j}\right|a_{n,k}}{a_{n+m,l}}\leq c_m \]
for every $n\in\N_0$. For arbitrary $m\in\N_0$, by \eqref{seminorm Bm, p finite} and \eqref{seminorm Bm, p infty}, respectively, a straight forward calculation gives $\Vert B_w^m (x)\Vert_{k,p}\leq c_m\Vert x\Vert_{k,p}$.
Thus, $B_w$ is topologizable.
\end{proof}

\begin{prop}\label{prop: power bounded}
	For a K\"othe matrix $A$, $p\in [1,\infty]\cup\{0\}$, and $w\in\omega$ the following hold.
	\begin{itemize}
		\item[(a)] If $B_w$ is continuous on $\lambda_p(A)$ the following are equivalent.
		\begin{itemize}
			\item[(i)] $B_w$ is power bounded on $\lambda_p(A)$.
			\item[(ii)] For every $k\in\N_0$ there is $l\in \N_0$ such that
			\begin{equation}\label{eq: characterization power boundedness B}
				\sup_{n\in\N_0,m\in\N}\frac{\left|\prod_{j=0}^{m-1} w_{n+j}\right|a_{n,k}}{a_{n+m,l}}<\infty.
			\end{equation}
		\end{itemize} 
		\item[(b)] If $F_w$ is continuous on $\lambda_p(A)$ the following are equivalent.
		\begin{itemize}
			\item[(i)] $F_w$ is power bounded on $\lambda_p(A)$.
			\item[(ii)] For every $k\in\N_0$ there is $l\in \N_0$ such that
			\begin{equation}\label{eq: characterization powerboundedness B}
				\sup_{n\in\N_0, m\in\N}\frac{\left|\prod_{j=1}^{m} w_{n+j}\right|a_{n+m,k}}{a_{n,l}}<\infty.
			\end{equation}
		\end{itemize} 
	\end{itemize}
\end{prop}

\begin{proof}
Again, we only present the proof of part (a) since the proof of (b) is mutatis mutandis the same. Hence, let $B_w$ be continuous on $\lambda_p(A)$ and  
assume that $B_w$ is power bounded. Then given $k\in\N_0$ there is $l\in\N_0$ and $c>0$ with
\[ \Vert B^m x\Vert_k\leq c \Vert x\Vert_l, \]
for every $x\in\lambda_p(B)$ and $m\in\N_0$. Evaluating this inequality for $x=e_{n+m}$ and using \eqref{seminorm Bm, p finite} and \eqref{seminorm Bm, p infty}, respectively, yields as in the proof of Proposition \ref{prop: topologizable} that (ii) holds.

Conversely, if (ii) is valid, fix $k\in \N_0$ and let $l\in\N_0$ be according to (ii) and let $c>0$ be such that
\[ \left|\prod_{j=0}^{m-1}w_{n+j}\right|a_{n,k}\leq c\, a_{n+m,l} \]
for every $n,m\in\N_0$. Using \eqref{seminorm Bm, p finite} and \eqref{seminorm Bm, p infty} it easily follows that
\[\forall\,x\in\lambda_p(A), m\in\N_0:\, \Vert B_w^m x\Vert_k\leq c\Vert x\Vert_l \]
so that $B$ is power bounded.
\end{proof}

\begin{remark}
	It should be noted that for $B_w$ and $F_w$ the properties of being correctly defined, continuous, topologizable, and power bounded on $\lambda_p(A)$ do not depend on the explicit value of $p\in[1,\infty]\cup\{0\}$.
\end{remark}

Next, we study Cesàro boundedness of weighted backward shifts $B_w$.

For $p\in[1,\infty)$ and $k\in\N_0$ it holds for $n\in\N$ and $x\in\lambda_p(A)$
\begin{equation}\label{seminorm B[n], p finite}
	\Vert B_w^{[n]} x\Vert_{k,p}^p=\sum_{j=0}^\infty\left|\frac{1}{n}\sum_{m=1}^n\left(\prod_{t=0}^{m-1}w_{j+t}\right)x_{j+m}a_{j,k}\right|^p
\end{equation}
while for $p=\infty$ we have
\begin{equation}\label{seminorm B[n], p infty}
	\Vert B_w^{[n]} x\Vert_{k,\infty}=\sup_{j\in\N_0}\left|\frac{1}{n}\sum_{m=1}^n\left(\prod_{t=0}^{m-1}w_{j+t}\right)x_{j+m}a_{j,k}\right|.
\end{equation}

\begin{prop}\label{prop:Cb for B}
	Let $A$ be a K\"othe matrix and let $p\in[1,\infty]\cup\{0\}$. Moreover, let $w\in\omega$ be such that the weighted backward shift $B_w$ is continuous on $\lambda_p(A)$. Consider the following conditions.
	\begin{itemize}
		\item[(i)] $(Cb_p^B)$ holds, where
		\begin{equation*}
			(Cb_p^B)=\begin{cases}\mbox{If }p\in[1,\infty):& \forall\,k\in\N_0\exists\,l\in\N_0:\,\sup_{r\in\N_0, n\in\N}\frac{\sum_{m=1}^n\left|\prod_{s=1}^{m}w_{r-s}\right|^p a_{r-m,k}^p}{n a_{r,l}^p}<\infty\\
				\mbox{If }p\in\{0,\infty\}:& \forall\,k\in\N_0\exists\,l\in\N_0:\,\sup_{r\in\N_0, n\in\N}\frac{\sum_{m=1}^n\left|\prod_{s=1}^{m}w_{r-s}\right| a_{r-m,k}}{n a_{r,l}}<\infty
			\end{cases}
		\end{equation*}
		\item[(ii)] $B_w$ is Cesàro bounded on $\lambda_p(A)$.
		\item[(iii)] For every $k\in\N_0$ there is $l\in\N_0$ such that
		$$\sup_{r\in\N_0,n\in\N}\frac{\left|\sum_{m=1}^{n}\left(\prod_{s=1}^{m}w_{r-s}\right)a_{r-m,k}\right|}{na_{r,l}}<\infty$$
	\end{itemize}
	Then, condition (i) implies (ii) and condition (ii) implies (iii).
\end{prop}

\begin{proof}
We first show that (ii) implies (iii). In order to do so, keeping in mind that by our convention $a_{s,k}=0$ for $s\in\Z\backslash\N_0$, we evaluate \eqref{seminorm B[n], p finite} and \eqref{seminorm B[n], p infty} for $x=e_r$ and obtain
$$\Vert B_w^{[n]} e_r\Vert_{k,p}^p=\frac{1}{n^p}\left|\sum_{m=1}^n\left(\prod_{t=0}^{m-1}w_{r-m+t}\right)a_{r-m,k}\right|^p$$
in case $p\in[1,\infty)$, respectively
$$\Vert B_w^{[n]} e_r\Vert_{k,\infty}=\left|\frac{1}{n}\sum_{m=1}^n\left(\prod_{t=0}^{m-1}w_{r-m+t}\right)a_{r-m,k}\right|$$
for $p=\infty$. Thus, Cesàro boundedness of $B_w$ on $\lambda^p(A)$ implies that for $k\in\N_0$ there are $l\in\N_0$ and $c>0$ such that
$$\forall\,r\in\N_0, n\in\N:\,c\,a_{r,l}\geq \left|\frac{1}{n}\sum_{m=1}^n\left(\prod_{t=0}^{m-1}w_{r-m+t}\right)a_{r-m,k}\right|=\left|\frac{1}{n}\sum_{m=1}^n\left(\prod_{s=1}^{m}w_{r-s}\right)a_{r-m,k}\right|$$
which proves (iii).

Now, assume that (i) holds for $p\in[1,\infty)$. Fix $k\in\N_0$ and choose $l\in\N_0$ according to $(Cb_p^B)$. For $n\in\N$ and $x\in\lambda_p(A)$ it follows with the convexity of the function $[0,\infty)\rightarrow\R, t\mapsto t^p$ 
\begin{eqnarray*}
	\Vert B_w^{[n]}x\Vert_{k,p}^p&\leq& \sum_{j=0}^\infty\frac{1}{n}\sum_{m=1}^n\left(\left|\prod_{t=0}^{m-1}w_{j+t}\right||x_{j+m}| a_{j,k}\right)^p\\
	&=&\sum_{r=0}^\infty\left(\frac{1}{n}\sum_{m=1}^n\left|\prod_{t=0}^{m-1}w_{r-m+t}\right|^p a_{r-m,k}^p\right)|x_r|^p\\
	&=&\sum_{r=0}^\infty\frac{\sum_{m=1}^n\left|\prod_{s=1}^{m}w_{r-s}\right|^p a_{r-m,k}^p}{n a_{r,l}^p}(|x_r|a_{r,l})^p\\
	&\leq&\sup_{r\in\N_0, n\in\N}\frac{\sum_{m=1}^n\left|\prod_{s=1}^{m}w_{r-s}\right|^p a_{r-m,k}^p}{n a_{r,l}^p} \Vert x\Vert_{l,p}.
\end{eqnarray*}
Since by assumption $(Cb_p^B)$ the supremum is finite we conclude that $B_w$ is Cesàro bounded on $\lambda_p(A)$.

Finally, if (i) holds for $p\in\{0,\infty\}$, let $k\in\N_0$ be arbitrary. Choosing $l\in\N_0$ according to $(Cb_p^B)$ implies for $n\in\N$ and $x\in\lambda_p(A)$
\begin{eqnarray*}
	\Vert B_w^{[n]}x\Vert_{k,p}&=&\sup_{j\in\N_0} \frac{1}{n}\left|\sum_{m=1}^n\left(\prod_{t=0}^{m-1}w_{j+t}\right)x_{j+m}\right|a_{j,k}\\
	&\leq&\sup_{j\in\N_0} \frac{\sum_{m=1}^n\left|\prod_{t=0}^{m-1}w_{j-m+t}\right|a_{j-m,k}}{n a_{j,l}}|x_j| a_{j,l}\\
	&=&\sup_{j\in\N_0} \frac{\sum_{m=1}^n\left|\prod_{s=1}^{m}w_{j-s}\right|a_{j-m,k}}{n a_{j,l}}|x_j| a_{j,l}\\
	&\leq&\sup_{r\in\N_0,n\in\N}  \frac{\sum_{m=1}^n\left|\prod_{s=1}^{m}w_{r-s}\right|a_{r-m,k}}{n a_{r,l}}\Vert x\Vert_{l,p}\\
\end{eqnarray*}
Since again by assumption $(Cb_p^B)$ the supremum is finite we conclude that $B_w$ is Cesàro bounded on $\lambda_p(A)$ for $p\in\{0,\infty\}$ which proves the proposition.
\end{proof}

As an immediate consequence of the previous proposition we obtain the following result.

\begin{coro}\label{cor: Cb for non-negative weights}
	Let $A$ be a K\"othe matrix and let $w\in\omega$ be such that $w_n\geq 0$, $n\in\N_0$, and such that the weighted backward shift $B_w$ is continuous on $\lambda_p(A)$. Then, the following are equivalent.
	\begin{itemize}
		\item[(i)] $B_w$ is Cesàro bounded on any/each of the spaces $\lambda_0(A),\lambda_1(A)$, or $\lambda_\infty(A)$.
		\item[(ii)] For every $k\in\N_0$ there is $l\in\N_0$ such that
		$$\sup_{r\in\N_0, n\in\N}\frac{\sum_{m=1}^n\left(\prod_{s=1}^{m}w_{r-s}\right) a_{r-m,k}}{n a_{r,l}}<\infty.$$
	\end{itemize}
	If additionally, $w_n\leq 1$, $n\in\N_0$, and if the K\"othe matrix $A=(a_{n,k})_{n,k\in\N_0}$ satisfies that $(a_{n,k})_{n\in\N_0}$ is increasing for every $k\in\N_0$, the equivalent conditions (i) and (ii) are also equivalent to condition (iii):
	\begin{itemize}
		\item[(iii)] $B_w$ is power bounded on any/each of the spaces $\lambda_0(A),\lambda_1(A)$, or $\lambda_\infty(A)$.
	\end{itemize}
\end{coro}

\begin{proof}
	That (i) and (ii) are equivalent follows immediately from Proposition \ref{prop:Cb for B}. Now, assume that the additional hypothesis on $w$ and $A$ hold. We fix $k\in\N_0$. From (ii) we conclude the existence of $l\in\N_0$ such that
	\begin{eqnarray*}
		\infty&>&\sup_{r\in\N_0, n\in\N}\frac{\sum_{m=1}^n\left(\prod_{s=1}^{m}w_{r-s}\right) a_{r-m,k}}{n a_{r,l}}\geq\sup_{r\in\N_0, n\in\N}\frac{\sum_{m=1}^n\left(\prod_{s=1}^{n}w_{r-s}\right) a_{r-n,k}}{n a_{r,l}}\\
		&=&\sup_{r\in\N_0, n\in\N}\frac{\left(\prod_{s=1}^{n}w_{r-n+n-s}\right) a_{r-n,k}}{a_{r-n+n,l}}=\sup_{j\in\N_0, n\in\N}\frac{\left(\prod_{t=0}^{n-1}w_{j+t}\right) a_{j,k}}{a_{j+n,l}}.
	\end{eqnarray*}
	Proposition \ref{prop: power bounded} (a) implies that $B_w$ is power bounded. On the other hand, every power bounded operator is Ces\`aro bounded which proves the statement.
\end{proof}

\begin{coro}\label{cor: sufficiency for mean ergodicity of B, p finite}
	Let $A$ be a K\"othe matrix and let $p\in (1,\infty)$. Let $w\in\omega$ be such that $w_n\geq 0$, $n\in\N_0$, and such that the weighted backward shift $B_w$ is continuous on $\lambda_p(A)$. Consider the following conditions.
	\begin{itemize}
		\item[(i)] For every $k\in\N_0$ there is $l\in\N_0$ such that
		$$\sup_{r\in\N_0, n\in\N}\frac{\sum_{m=1}^n\left(\prod_{s=1}^{m}w_{r-s}\right)^p a_{r-m,k}^p}{n a_{r,l}^p}<\infty.$$
		\item[(ii)] $B_w$ is mean ergodic on $\lambda_p(A)$.
		\item[(iii)]  For every $k\in\N_0$ there is $l\in\N_0$ such that
		$$\sup_{r\in\N_0, n\in\N}\frac{\sum_{m=1}^n\left(\prod_{s=1}^{m}w_{r-s}\right) a_{r-m,k}}{n a_{r,l}}<\infty.$$
	\end{itemize}
	Then, condition (i) implies (ii) and condition (ii) implies (iii).
\end{coro}

\begin{proof}
	Assume that (i) holds. By Proposition \ref{prop:Cb for B}, $B_w$ is Cesàro bounded on $\lambda_p(A)$, i.e.\ $\{B_w^{[n]};n\in\N\}$ is equicontinuous on $\lambda_p(A)$. Because $\frac{1}{n}B_w^n=B_w^{[n]}-\frac{n-1}{n}B_w^{[n-1]}$ the set $\{\frac{1}{n}B_w^n;n\in\N\}$ is equicontinuous on $\lambda_p(A)$, too. Moreover, $\mbox{span}\{e_r;r\in\N_0\}$ is dense in $\lambda_p(A)$ and obviously for each $x\in \mbox{span}\{e_r;r\in\N_0\}$ it holds $\frac{1}{n}B_w^n x=0$ for sufficiently large $n$. From equicontinuity it follows that $\left(\frac{1}{n}B_w^n x\right)_{n\in\N}$ tends to 0 for every $x\in\lambda_p(A)$. Since $\lambda_p(A)$ is reflexive (see, e.g.\ \cite[Proposition 27.3]{MeVo1997}) Theorem \ref{theo: mean ergodic} implies that $B_w$ is mean ergodic on $\lambda_p(A)$.
	
	On the other hand, if (ii) holds, refereing to Theorem \ref{theo: mean ergodic} again, it follows that $B_w$ is Cesàro bounded on $\lambda_p(A)$ so that (iii) follows from Proposition \ref{prop:Cb for B}.
\end{proof}

Before we continue with a result characterizing mean ergodicity of $B_w$ on $\lambda_0(A)$ and $\lambda_1(A)$, we recall that by the Dieudonn\'e-Gomes Theorem (see e.g.\ \cite[Theorem 27.9]{MeVo1997}) $\lambda_1(A)$ is reflexive if and only if $\lambda_p(A)$ is a Montel space for all/some $p\in [1,\infty]$ if and only if $\lambda_0(A)=\lambda_\infty(A)$ and that these properties are equivalent to
\begin{equation}\label{eq: characteristic for Montel}
	\forall\,I\subseteq\N_0\mbox{ infinite}, k\in\N_0\,\exists\,l\in\N_0:\,\inf_{n\in I}\frac{a_{n,k}}{a_{n,l}}=0.
\end{equation}
Moreover, recall that the above properties are equivalent to $\lambda_0(A)$ being a Montel space (see e.g.\ \cite[Proposition 27.15]{MeVo1997}).

\begin{theo}\label{theo: mean ergodicity for the Montel case}
	Let $w\in\omega$ be such that $w_n\geq 0$, $n\in\N_0$. Moreover, let $A$ be a K\"othe matrix such that $\lambda^p(A)$ is a Montel space and such that the weighted backward shift $B_w$ is continuous on $\lambda_p(A)$ for some/all $p\in[1,\infty]\cup\{0\}$. Then, the following are equivalent.
	\begin{itemize}
		\item[(i)] $B_w$ is Cesàro bounded on one /each of the spaces $\lambda_0(A)$ or $\lambda_1(A)$.
		\item[(ii)] For every $k\in\N_0$ there is $l\in\N_0$ such that
		$$\sup_{r\in\N_0, n\in\N}\frac{\sum_{m=1}^n\left(\prod_{s=1}^{m}w_{r-s}\right) a_{r-m,k}}{n a_{r,l}}<\infty.$$
		\item[(iii)] $B_w$ is (uniformly) mean ergodic on either of the spaces $\lambda_0(A)$ or $\lambda_1(A)$.
	\end{itemize}
	If additionally, $w_n\leq 1$, $n\in\N_0$, and if the K\"othe matrix $A=(a_{n,k})_{n,k\in\N_0}$ satisfies that $(a_{n,k})_{n\in\N_0}$ is increasing for every $k\in\N_0$, the above conditions are also equivalent to condition (iv):
	\begin{itemize}
		\item[(iv)] $B_w$ is power bounded on either of the spaces $\lambda_0(A)$ and $\lambda_1(A)$.
	\end{itemize}
\end{theo}

\begin{proof}
	The equivalence of (i) and (ii) holds by Corollary \ref{cor: Cb for non-negative weights}. Moreover, since Fr\'echet spaces are barrelled, mean ergodic operators on Fr\'echet spaces are Cesàro bounded, so that condition (iii) implies (i).
	
	Next, if (ii) holds, refering to Theorem \ref{theo:ergodicity and transposed} instead of Theorem \ref{theo: mean ergodic}, it follows as in the proof of Corollary \ref{cor: sufficiency for mean ergodicity of B, p finite} that $B_w$ is uniformly mean ergodic on $\lambda_0(A)$ and $\lambda_1(A)$.
	
	In case that the additional properties of $A$ and $w$ are satisfied it follows from Corollary \ref{cor: Cb for non-negative weights} that (i) and (iv) are equivalent.
\end{proof}

In the remainder of this section we study Cesàro boundedness and mean ergodicity of weighted forward shifts $F_w$. For $p\in[1,\infty)$ and $k\in\N_0$ it holds for $n\in\N$ and $x\in\lambda_p(A)$
\begin{equation}\label{seminorm F[n], p finite}
	\Vert F_w^{[n]} x\Vert_{k,p}^p=\sum_{j=0}^\infty\left|\frac{1}{n}\sum_{m=1}^n\left(\prod_{t=0}^{m-1}w_{j-t}\right)x_{j-m}a_{j,k}\right|^p
\end{equation}
while for $p=\infty$ we have
\begin{equation}\label{seminorm F[n], p infty}
	\Vert F_w^{[n]} x\Vert_{k,\infty}=\sup_{j\in\N_0}\left|\frac{1}{n}\sum_{m=1}^n\left(\prod_{t=0}^{m-1}w_{j-t}\right)x_{j-m}a_{j,k}\right|
\end{equation}

Using \eqref{seminorm F[n], p finite} and \eqref{seminorm F[n], p infty} instead of \eqref{seminorm B[n], p finite} and \eqref{seminorm B[n], p infty}, respectively, the proof of the next proposition is done precisely as the one of Proposition \ref{prop:Cb for B} and is therefore omitted.

\begin{prop}\label{prop:Cb for F}
	Let $A$ be a K\"othe matrix and let $p\in[1,\infty]\cup\{0\}$. Moreover, let $w\in\omega$ be such that the weighted forward shift $F_w$ is continuous on $\lambda_p(A)$. Consider the following conditions.
	\begin{itemize}
		\item[(i)] $(Cb_p^F)$ holds, where
		\begin{equation*}
			(Cb_p^F)=\begin{cases}\mbox{If }p\in[1,\infty):& \forall\,k\in\N_0\exists\,l\in\N_0:\,\sup_{r\in\N_0,n\in\N}\frac{\sum_{m=1}^{n}\left|\prod_{s=1}^{m}w_{r+s}\right|^p a_{r+m,k}^p}{na_{r,l}^p}<\infty\\
				\mbox{If }p\in\{0,\infty\}:& \forall\,k\in\N_0\exists\,l\in\N_0:\,\sup_{r\in\N_0,n\in\N}\frac{\sum_{m=1}^{n}\left|\prod_{s=1}^{m}w_{r+s}\right|a_{r+m,k}}{n a_{r,l}}<\infty
			\end{cases}
		\end{equation*}
		\item[(ii)] $F_w$ is Cesàro bounded on $\lambda_p(A)$.
		\item[(iii)] For every $k\in\N_0$ there is $l\in\N_0$ such that
		$$\sup_{r\in\N_0,n\in\N}\frac{\left|\sum_{m=1}^{n}\left(\prod_{s=1}^{n}w_{r+s}\right)a_{r+m,k}\right|}{na_{j,l}}<\infty$$
	\end{itemize}
	Then, condition (i) implies (ii) and condition (ii) implies (iii).
\end{prop}

Analogously as Corollary \ref{cor: Cb for non-negative weights} one obtains the next result.

\begin{coro}\label{cor: Cb for non-negative weights for F}
	Let $A$ be a K\"othe matrix and let $w\in\omega$ be such that $w_n\geq 0$, $n\in\N_0$, and such that the weighted forward shift $F_w$ is continuous on $\lambda_p(A)$. Then, the following are equivalent.
	\begin{itemize}
		\item[(i)] $F_w$ is Cesàro bounded on any/each of the spaces $\lambda_0(A),\lambda_1(A)$, or $\lambda_\infty(A)$.
		\item[(ii)] For every $k\in\N_0$ there is $l\in\N_0$ such that
		$$\sup_{r\in\N_0,n\in\N}\frac{\sum_{m=1}^{n}\left(\prod_{s=1}^{m}w_{r+s}\right)a_{r+m,k}}{na_{r,l}}<\infty.$$
	\end{itemize}
	If additionally, $w_n\leq 1$, $n\in\N_0$, and if the K\"othe matrix $A=(a_{n,k})_{n,k\in\N_0}$ satisfies that $(a_{n,k})_{n\in\N_0}$ is decreasing for every $k\in\N_0$, the equivalent conditions (i) and (ii) are also equivalent to condition (iii):
	\begin{itemize}
		\item[(iii)] $F_w$ is power bounded on any/each of the spaces $\lambda_0(A),\lambda_1(A)$.
	\end{itemize}
\end{coro}

From \eqref{seminorm Fm, p finite} and \eqref{seminorm Fm, p infty} it follows that for $x\in\mbox{span}\{e_r; r\in\N_0\}$ we have $\lim_{n\rightarrow\infty}\frac{1}{n}F_w^n x=0$ in $\lambda_p(A)$ if and only if
$$\forall\,r,k\in\N_0:\,\lim_{n\rightarrow\infty}\frac{\left|\prod_{s=1}^n w_{r+s}\right|a_{r+n,k}}{n}=0.$$
Repeating the arguments from the proofs of Corollary \ref{cor: sufficiency for mean ergodicity of B, p finite} and Theorem \ref{theo: mean ergodicity for the Montel case}, we conclude the next results.

\begin{coro}\label{cor: sufficiency for mean ergodicity of F, p finite}
	Let $A$ be a K\"othe matrix and let $p\in (1,\infty)$. Let $w\in\omega$ be such that $w_n\geq 0$, $n\in\N_0$, and such that the weighted foward shift $F_w$ is continuous on $\lambda_p(A)$. Consider the following conditions.
	\begin{itemize}
		\item[(i)] For every $k\in\N_0$ we have
		$$\forall\,r\in\N_0:\,\lim_{n\rightarrow\infty}\frac{\left(\prod_{s=1}^n w_{r+s}\right)a_{r+n,k}}{n}=0$$
		and there is $l\in\N_0$ such that
		$$\sup_{r\in\N_0,n\in\N}\frac{\sum_{m=1}^{n}\left(\prod_{s=1}^{m}w_{r+s}\right)^p a_{r+m,k}^p}{na_{r,l}^p}<\infty.$$
		\item[(ii)] $F_w$ is mean ergodic on $\lambda_p(A)$.
		\item[(iii)]  For every $k\in\N_0$ we have
		$$\forall\,r\in\N_0:\,\lim_{n\rightarrow\infty}\frac{\left(\prod_{s=1}^n w_{r+s}\right)a_{r+n,k}}{n}=0$$
		and there is $l\in\N_0$ such that
		$$\sup_{r\in\N_0,n\in\N}\frac{\sum_{m=1}^{n}\left(\prod_{s=1}^{m}w_{r+s}\right) a_{r+m,k}}{na_{r,l}}<\infty.$$
	\end{itemize}
	Then, condition (i) implies (ii) and condition (ii) implies (iii).
\end{coro}

\begin{theo}\label{theo: mean ergodicity of F for the Montel case}
	Let $w\in\omega$ be such that $w_n\geq 0$, $n\in\N_0$. Moreover, let $A$ be a K\"othe matrix such that $\lambda^p(A)$ is a Montel space and such that the weighted backward shift $F_w$ is continuous on $\lambda_p(A)$ for some/all $p\in[1,\infty]\cup\{0\}$. Then, the following are equivalent.
	\begin{itemize}
		\item[(i)] $F_w$ is (uniformly) mean ergodic on any/each of the spaces $\lambda_0(A)$ or $\lambda_1(A)$.
		\item[(ii)] For every $k\in\N_0$ there is $l\in\N_0$ such that
		$$\sup_{r\in\N_0,n\in\N}\frac{\sum_{m=1}^{n}\left(\prod_{s=1}^{m}w_{r+s}\right)a_{r+m,k}}{na_{r,l}}<\infty$$
		and it holds
		$$\forall\,r\in\N_0:\,\lim_{n\rightarrow\infty}\frac{\left|\prod_{s=1}^n w_{r+s}\right|a_{r+n,k}}{n}=0.$$
	\end{itemize}
\end{theo}

\section{Ergodic properties of weighted shifts on power series spaces and related results}\label{sec: power series spaces}

In this section we evaluate the results from the previous one for the special case of power series spaces and we apply these in order to study dynamical properties of some classical operators on certain function spaces. For an exponent sequence $\alpha=(\alpha_n)_{n\in\N_0}$ we have for the corresponding power series spaces $\Lambda_\infty(\alpha)=\lambda^1(A_\infty(\alpha))$ and $\Lambda_0(\alpha)=\lambda_1(A_0(\alpha))$ with the K\"othe matrices $A_\infty(\alpha)=\left(\exp(k\alpha_n)\right)_{k,n\in\N_0}$ and $A_0(\alpha)=\left(\exp\left(-\frac{\alpha_n}{k+1}\right)\right)_{k,n\in\N_0}$, respectively. In particular, property \eqref{eq: characteristic for Montel} is satisfied so that $\Lambda_\infty(\alpha)$ and $\Lambda_0(\alpha)$ are both Montel spaces.

\begin{prop}\label{prop: topologizable on power series spaces}
	Let $\alpha=\left(\alpha_n\right)_{n\in\N_0}$ be an exponent sequence and let $w\in\omega$. Then the following hold.
	\begin{itemize}
		\item[(i)] Assume that $B_w$ is continuous in $\Lambda_\infty(\alpha)$. Then, $B_w$ is topologizable on $\Lambda_\infty(\alpha)$ if and only if \begin{equation}\label{eq: topologizability of backward on infinite power series space}
			\sup_{m\in\N}\limsup_{n\rightarrow\infty}\frac{\ln\left|\prod_{j=0}^{m-1}w_{n+j}\right|}{\alpha_{n+m}}<\infty.
		\end{equation}
		\item[(ii)] Assume that $F_w$ is continuous in $\Lambda_\infty(\alpha)$. Then, $F_w$ is topologizable on $\Lambda_\infty(\alpha)$ if and only if \begin{equation}\label{eq: topologizability of forward on infinite power series space}
			\forall\,k\in\N_0:\, \sup_{m\in\N}\limsup_{n\rightarrow\infty}\frac{\ln\left|\prod_{j=1}^m w_{n+j}\right|+k\alpha_{n+m}}{\alpha_n}<\infty.
		\end{equation}
		\item[(iii)]  Assume that $B_w$ is continuous in $\Lambda_0(\alpha)$. If $B_w$ is topologizable on $\Lambda_0(\alpha)$ it holds
		\begin{equation}\label{eq: topologizability of backward on finite power series space}
			\sup_{m\in\N}\limsup_{n\rightarrow\infty}\frac{ \ln\left|\prod_{j=0}^{m-1}w_{n+j}\right|}{\alpha_{n+m}}\leq 0.
		\end{equation}
		If additionally \ $\sup_{n\in\N}(\alpha_{n+1}-\alpha_n)<\infty$ holds, condition \eqref{eq: topologizability of backward on finite power series space} gives that $B_w$ is topologizable on $\Lambda_0(\alpha)$.
		\item[(iv)]  Assume that $F_w$ is continuous in $\Lambda_0(\alpha)$. Then $F_w$ is topologizable on $\Lambda_0(\alpha)$ if and only if
		\begin{equation}\label{eq: topologizability of forward on finite power series space}
			\sup_{m\in\N}\limsup_{n\rightarrow\infty}\frac{ \ln\left|\prod_{j=1}^mw_{n+j}\right|}{\alpha_{n+m}}\leq 0.
		\end{equation}
	\end{itemize}
\end{prop}

\begin{proof}
	In order to prove (i), we observe that by Proposition \ref{prop: topologizable} $B_w$ is topologizable on $\Lambda_\infty(\alpha)$ if and only if for every $k\in\N_0$ there is $l\in\N_0$ such that
	\begin{equation}\label{eq: topologizability of backward on infinite power series space 1}
		 \forall\,m\in\N\,\exists\, C>0\,\forall\,n\in\N_0:\,\ln\left|\prod_{j=0}^{m-1}w_{n+j}\right|\leq C+l\alpha_{n+m}-k\alpha_n.
	\end{equation}
	Thus, if $B_w$ is topologizable, evaluating \eqref{eq: topologizability of backward on infinite power series space 1} for $k=0$ implies \eqref{eq: topologizability of backward on infinite power series space}. On the other hand, assuming \eqref{eq: topologizability of backward on infinite power series space} let $k\in\N_0$ be fixed. We set $R=\sup_{m\in\N}\limsup_{n\rightarrow\infty}\frac{\ln\left|\prod_{j=0}^{m-1}w_{n+j}\right|}{\alpha_{n+m}}$. For each $m\in\N_0$ there is $N_m\in\N$ such that
	$$\forall\,n\geq N_m:\,\ln\left|\prod_{j=0}^{m-1}w_{n+j}\right|+k\alpha_n\leq (R+1)k\alpha_{n+m}+k\alpha_n\leq (R+2)k\alpha_{n+m}.$$
	Thus \eqref{eq: topologizability of backward on infinite power series space 1} holds true for $l\in\N_0$ with $l\geq (R+2)k$ so that $B_w$ is topologizable on $\Lambda_\infty(\alpha)$.
	
	For the proof of (ii) we observe that by Proposition \ref{prop: topologizable} topologizability of $F_w$ on $\Lambda_\infty(\alpha)$ precisely when for every $k\in\N_0$ there is $l\in\N_0$ with
	$$\forall\,m\in\N\,\exists\,C>0\,\forall\,n\in\N_0:\,\ln\left|\prod_{j=1}^mw_{n+j}\right|+k\alpha_{n+m}\leq C +l\alpha_n.$$
	From this one easily derives that \eqref{eq: topologizability of forward on infinite power series space} is equivalent to the topologizability of $F_w$ on $\Lambda_\infty(\alpha)$ so that (ii) follows.
	
	For the proof of (iii) we apply Proposition \ref{prop: topologizable} to see that $B_w$ is topologizable on $\Lambda_0(\alpha)$ if and only if for every $k\in\N_0$ there is $l\in\N_0$ with
	\begin{equation}\label{eq: topologizability of backward on finite power series space 1}
		\forall\,m\in\N\,\exists\,C>0\,\forall\,n\in\N_0:\, \ln\left|\prod_{j=0}^{m-1} w_{n+j}\right|+\frac{\alpha_{n+m}}{l+1}\leq C +\frac{\alpha_n}{k+1}.
	\end{equation}
	Thus, if $B_w$ is topologizable on $\Lambda_0(\alpha)$, evaluating the above condition for $k=0$ implies \eqref{eq: topologizability of backward on finite power series space}. On the other hand, under the assumption that \eqref{eq: topologizability of backward on finite power series space} holds, for fixed $k\in\N_0$ and arbitrary $m\in\N_0$ there is $N_m\in\N$ such that
	\begin{eqnarray*}
		\forall\,n\geq N_m:\,\ln\left|\prod_{j=0}^{m-1} w_{n+j}\right|&\leq&\left(\frac{1}{k+1}-\frac{1}{k+2}\right)\alpha_{n+m}=\frac{\alpha_n}{k+1}+\frac{\sum_{j=0}^{m-1}(\alpha_{n+j+1}-\alpha_{n+j})}{k+1}-\frac{\alpha_{n+m}}{k+2}\\
		&\leq&\frac{m}{k+1}\sup_{r\in\N}(\alpha_{r+1}-\alpha_r)+\frac{\alpha_n}{k+1}-\frac{\alpha_{n+m}}{k+2}.
	\end{eqnarray*} 
	Thus, under the additional hypothesis that $\sup_{r\in\N}(\alpha_{r+1}-\alpha_r)<\infty$ we see that \eqref{eq: topologizability of backward on finite power series space 1} is satisfied for $l=k+1$ so that $F_w$ is topologizable. This proves (iii).
	
	Finally, for the proof of (iv), by Proposition \ref{prop: topologizable}, $F_w$ is topologizable on $\Lambda_0(\alpha)$ if and only if for arbitrary $k\in\N_0$ there is $l\in\N_0$ with
	\begin{equation}\label{eq: topologizability of forward on finite power series space 1}
		\forall\,m\in\N\,\exists\,C>0\,\forall\,n\in\N_0:\,\ln\left|\prod_{j=1}^m w_{n+j}\right|+\frac{\alpha_n}{l+1}\leq C +\frac{\alpha_{n+m}}{k+1}.
	\end{equation}
	Hence, if $F_w$ is topologizable on $\Lambda_0(\alpha)$ by the above condition \eqref{eq: topologizability of forward on finite power series space} follows. On the other hand, if \eqref{eq: topologizability of forward on finite power series space} holds, for fixed $k\in\N_0$ we derive for arbitrary $m\in\N$ the existence of $N_m\in\N$ for which
	$$\forall\,n\geq N_m:\,\ln\left|\prod_{j=1}^m w_{n+j}\right|\leq\frac{\alpha_{n+m}}{2(k+1)}\leq\frac{\alpha_{n+m}}{k+1}-\frac{\alpha_n}{2(k+1)}.$$
	Thus, with $l=2k+1$ it follows
	$$\forall\,m\in\N\,\forall\,n\in\N_0:\,\ln\left|\prod_{j=1}^m w_{n+j}\right|+\frac{\alpha_n}{l+1}\leq\max\left\{ \ln\left|\prod_{j=1}^m w_{r+j}\right|+\frac{\alpha_r}{l+1} ;1\leq r\leq N_m\right\} +\frac{\alpha_{n+m}}{k+1}$$
	so that \eqref{eq: topologizability of forward on finite power series space 1} holds which proves (iv).
\end{proof}

In order to give a simplified condition of power boundedness for $B_w$ and $F_w$ on power series spaces we introduce the following notion. Given a matrix $\left(A_{n,m}\right)_{n\in\N_0,m\in\N}\in\R^{\N_0\times\N}$ we define the superior limit along $n+m$ as
\[ \limsup_{n+m\to \infty} A_{n,m}:= \inf_{k\in\N}\left\{ \sup_{n+m\geq k}\left\{ A_{n,m} \right\}\right\}. \]

\begin{prop}\label{prop: power bounded on power series spaces}
	Let $\alpha=\left(\alpha_n\right)_{n\in\N_0}$ be an exponent sequence with $\alpha_0>0$ and let $w\in\omega$. Then the following hold.
	\begin{itemize}
		\item[(i)] Assume that $B_w$ is continuous in $\Lambda_\infty(\alpha)$. $B_w$ is power bounded on $\Lambda_\infty(\alpha)$ if and only if
		\begin{equation}\label{eq: power boundedness of backward on infinite power series space}
			\sup_{n\in\N_0,m\in\N}\frac{\ln\left|\prod_{j=0}^{m-1}w_{n+j}\right|}{\alpha_{n+m}}<\infty.
		\end{equation}
		\item[(ii)] Assume that $F_w$ is continuous on $\Lambda_\infty(\alpha)$. Then $F_w$ is power bounded on $\Lambda_\infty(\alpha)$ if and only if
		\begin{equation}\label{eq: power boundedness of forward on infinite power series space}
			\forall\,k\in\N_0:\, \sup_{n\in\N_0,m\in\N}\frac{\ln\left|\prod_{j=1}^{m}w_{n+j}\right|+k\alpha_{n+m}}{\alpha_n}<\infty.
		\end{equation}
		\item[(iii)] Assume that $B_w$ is continuous in $\Lambda_0(\alpha)$. Then $B_w$ is power bounded on $\Lambda_0(\alpha)$ if and only if
		\begin{equation}\label{eq: power boundedness of backward on finite power series space}
		\forall\,k\in\N_0:\,
		\limsup_{n+m\rightarrow\infty}\frac{\ln\left|\prod_{j=0}^{m-1}w_{n+j}\right|-\frac{\alpha_n}{k+1}}{\alpha_{n+m}}<0
		\end{equation}
		\item[(iv)] Assume that $F_w$ is continuous on $\Lambda_0(\alpha)$. $F_w$ is power bounded on $\Lambda_0(\alpha)$ if and only if
		\begin{equation}\label{eq: power boundedness of forward on finite power series space}
			\limsup_{n+m\rightarrow\infty}\frac{ \ln\left|\prod_{j=1}^{m}w_{n+j}\right|}{\alpha_{n+m}}\leq 0.
		\end{equation}
	\end{itemize}
\end{prop}
\begin{proof}
	Let $B_w$ be continuous on $\Lambda_\infty(\alpha)$. By  Proposition \ref{prop: power bounded} $B_w$ is power bounded on $\Lambda_\infty(\alpha)$ if and only if for each $k\in\N_0$ there are $l\in\N_0$ and $C>0$ with
	\begin{equation}\label{eq: power boundedness of backward on infinite power series space 1}
		\forall n\in\N_0, m\in\N:  \left|\prod_{j=0}^{m-1} w_{n+j}\right| \leq C\exp(l\alpha_{n+m}-k\alpha_{n}).
	\end{equation}
	Therefore, whenever $B_w$ is power bounded, evaluating the previous condition for $k=0$ implies \eqref{eq: power boundedness of backward on infinite power series space}. Conversely, if \eqref{eq: power boundedness of backward on infinite power series space} is satisfied, for a fixed $k\in\N_0$ we can find $c>0$ such that
	\[ \forall n\in\N_0, m\in\N:  \ln\left|\prod_{j=0}^{m-1} w_{n+j}\right|\leq ck\alpha_{n+m}. \]
	Then, for every $n\in\N_0, m\in\N$ we conclude
	\[ \left|\prod_{j=0}^{m-1} w_{n+j}\right|\exp(k\alpha_n)\leq \exp(ck\alpha_{n+m}+k\alpha_{n})\leq \exp(k(c+1)\alpha_{n+m}) \]
	so that \eqref{eq: power boundedness of backward on infinite power series space 1} holds true for every $l\in\N$ with $l\geq k(c+1)$ and $C=1$ so that $B_w$ is power bounded. This proves (i).
	
	In order to prove (ii) we apply Proposition \ref{prop: power bounded} to deduce that $F_w$ is power bounded on $\Lambda_0(\alpha)$ if and only if for each $k\in\N_0$ there are $l\in\N_0$ and $C>0$ such that
	$$\forall\,n\in\N_0, m\in\N:\, \ln\left|\prod_{j=1}^m w_{n+j}\right|+k\alpha_{n+m}\leq C+ l\alpha_n$$
	which is easily seen to be equivalent to \eqref{eq: power boundedness of forward on infinite power series space}.
	
	By Proposition \ref{prop: power bounded} $B_w$ is power bounded on $\Lambda_0(\alpha)$ if and only if the condition given in (a) (ii) holds which is clearly equivalent to for each $k\in\N_0$ there are $l\in\N_0$ and $C>0$ such that
	\begin{equation}\label{eq: power boundedness of backward on finite power series space 1}
		\forall\,n\in\N_0, m\in\N:\,\ln\left|\prod_{j=0}^{m-1} w_{n+j}\right| -\frac{\alpha_n}{k+1}\leq C- \frac{\alpha_{n+m}}{l+1} .
	\end{equation}
	Hence, if $B_w$ is power bounded on $\Lambda_0(\alpha)$ it follows from the above condition and the fact that $(\alpha_n)_{n\in\N_0}$ is increasing and tends to infinity, that for each $k\in\N_0$ there is $N_k>0$ such that $C/\alpha_{n+m}\leq \frac{1}{2(l+1)}$ for every pair $(n,m)\in \N_0\times\N$ with $n+m\geq N_k$. This gives  
	\[ \frac{\ln\left|\prod_{j=0}^{m-1}w_{n+j}\right| -\frac{\alpha_n}{k+1}}{\alpha_{n+m}}\leq -\frac{1}{2(l+1)} \]
	whenever $n+m\geq N_k$ is satisfied. Therefore, condition \eqref{eq: power boundedness of backward on finite power series space} holds. 
	
	Now assume that \eqref{eq: power boundedness of backward on finite power series space} holds. Then for a fixed $k\in \N_0$ there are $N_k>0$ and $A_k>1$ such that
	\[ \frac{\ln\left|\prod_{j=0}^{m-1}w_{n+j}\right| -\frac{\alpha_n}{k+1}}{\alpha_{n+m}}\leq -\frac{1}{A_k(k+1)} \]
	is satisfied for every pair $(n,m)\in \N_0\times\N$ with $n+m\geq N_k$. Now, since $\{(n,m): n+m<N_k\}$ is finite this gives that condition \eqref{eq: power boundedness of backward on finite power series space 1} holds for every $l\geq A_k(k+1)$. This proves (iii).
	
	Finally, in order to prove (iv), by Proposition \ref{prop: power bounded} $F_w$ is power bounded on $\Lambda_0(\alpha)$ if and only if for every $k\in\N_0$ there are $l\in\N_0$ and $C>0$ such that
	\begin{equation}\label{eq: power boundedness of forward on finite power series space 1} 
		\forall\,n\in\N_0, m\in\N:\,\ln\left|\prod_{j=1}^n w_{n+j}\right|+\frac{\alpha_n}{l+1}\leq C +\frac{\alpha_{n+m}}{k+1}.
	\end{equation}
	 Now assume that $F_w$ is power bounded. For each $k\in\N_0$ there are thus $l\in\N_0$ and $C>0$ with 
	\[ \forall n\in\N_0, m\in\N:\,  \frac{\ln\left|\prod_{j=1}^{m} w_{n+j}\right|}{\alpha_{n+m}} \leq \frac{C}{\alpha_{n+m}}+\frac{1}{k+1}. \]
	Now since $(\alpha_n)_{n\in\N_0}$ tends to infinity, there is $N_k>0$ such that $C/\alpha_{n+m}\leq \frac{1}{k+1}$ for every pair $(n,m)\in \N_0\times\N$ with $n+m\geq N_k$. Thus the above inequality implies condition \eqref{eq: power boundedness of forward on finite power series space}.
	
	On the other hand, if \eqref{eq: power boundedness of forward on finite power series space} is satisfied, for a fixed $k\in N_0$ there is $N_k$ such that 
	\[ \forall n\in\N_0, m\in\N, n+m\geq N_k: \ln\left|\prod_{j=1}^{m} w_{n+j}\right|\leq \frac{\alpha_{n+m}}{2(k+1)}. \]
	Thus we obtain
	\[ \forall n\in\N_0,m\in\N, n+m\geq N_k: \ln\left|\prod_{j=1}^{m} w_{n+j}\right|+\frac{\alpha_n}{2(k+1)}\leq\frac{\alpha_n+\alpha_{n+m}}{2(k+1)}\leq\frac{\alpha_{n+m}}{k+1}\]
	since $\alpha$ is monotonically increasing.	Since the set $\{(n,m): n+m<N_k\}$ is finite, this gives that condition \eqref{eq: power boundedness of forward on finite power series space 1} holds with $l=2k+1$ so that $F_w$ is power bounded on $\Lambda_0(\alpha)$. This proves (iv).
\end{proof}

As an immediate consequence of Theorems \ref{theo: mean ergodicity for the Montel case} and \ref{theo: mean ergodicity of F for the Montel case} we obtain the next corollary.

\begin{coro}\label{cor: mean ergodicity on power series spaces}
	Let $\alpha$ be an exponent sequence and $w\in\omega$ with $w_n\geq 0$ for each $n\in\N_0$.
	\begin{itemize}
		\item[(i)] Assume that $B_w$ is continuous on $\Lambda_\infty(\alpha)$. $B_w$ is (uniformly) mean ergodic on $\Lambda_\infty(\alpha)$ if and only if for every $k\in\N_0$ there is $l\in\N_0$ such that
		$$\sup_{r,n\in\N, r\geq n}\frac{1}{n}\sum_{m=1}^n\left(\prod_{s=1}^m w_{r-s}\right)\exp\left(k\alpha_{r-m}-l\alpha_r\right)<\infty.$$
		
		\item[(ii)] Assume that $F_w$ is continuous on $\Lambda_\infty(\alpha)$. $F_w$ is (uniformly) mean ergodic on $\Lambda_\infty(\alpha)$ if and only if for every $k\in\N_0$
		$$\forall\,r\in\N_0:\,\lim_{n\rightarrow\infty}\frac{\left(\prod_{s=1}^n w_{r+s}\right)\exp(k\alpha_{r+n})}{n}=0$$
		and there is $l\in\N_0$ such that
		$$\sup_{r\in\N_0, n\in\N}\frac{1}{n}\sum_{m=1}^n\left(\prod_{s=1}^m w_{r+s}\right)\exp\left(k\alpha_{r+m}-l\alpha_r\right)<\infty.$$	
		
		\item[(iii)] Assume that $B_w$ is continuous on $\Lambda_0(\alpha)$. $B_w$ is (uniformly) mean ergodic on $\Lambda_0(\alpha)$ if and only if for every $k\in\N_0$ there is $l\in\N_0$ such that
		$$\sup_{r, n\in\N, r\geq n}\frac{1}{n}\sum_{m=1}^n\left(\prod_{s=1}^m w_{r-s}\right)\exp\left(\frac{\alpha_r}{l+1}-\frac{\alpha_{r-m}}{k+1}\right)<\infty.$$
		
		\item[(iv)] Assume that $F_w$ is continuous on $\Lambda_0(\alpha)$.$F_w$ is (uniformly) mean ergodic on $\Lambda_0(\alpha)$ if and only if for every $k\in\N_0$
		$$\forall\,r\in\N_0:\,\lim_{n\rightarrow\infty}\frac{\left(\prod_{s=1}^n w_{r+s}\right)\exp\left(-\frac{\alpha_{r+n}}{k+1}\right)}{n}=0$$
		and there is $l\in\N_0$ such that
		$$\sup_{r\in\N_0, n\in\N}\frac{1}{n}\sum_{m=1}^n\left(\prod_{s=1}^m w_{r+s}\right)\exp\left(\frac{\alpha_r}{l+1}-\frac{\alpha_{r+m}}{k+1}\right)<\infty.$$	
		\end{itemize}	
\end{coro}

Now we study ergodicity and related properties for some classical operators acting on the Fr\'echet spaces of holomorphic functions $H(\C)$ and $H(\D)$, where $\D$ denotes the open unit disk in the complex plane. Via the mapping $\Phi$ which assigns to a holomorphic function $f$ the sequence of its Taylor coefficients in the origin, $\Phi(f)=(\frac{f^{(n)}(0)}{n!})_{n\in\N_0}$, these spaces are topologically isomorphic to the power series space $\Lambda_\infty((n+1))_{n\in\N_0}$, respectively $\Lambda_0((n+1)_{n\in\N_0})$. For $G\in\{\C,\D\}$ let $\Delta_0$ be the continuous linear operator on $H(G)$ which  maps $f$ to the holomorphic function $\Delta_0(f)(z)= (f(z)-f(0))/z$, $z\in G\backslash\{0\}$, $\Lambda_0(f)(0)=f'(0)$. Additionally, we denote by $V$ the Volterra operator on $H(G)$ which takes $f\in H(G)$ to the holomorphic function $z\mapsto \int_0^z f(\zeta)\,d\zeta$, $z\in G$. As usual, $\frac{d}{dz}$ denotes the differentiation operator on $H(G)$ which maps $f$ to its derivative $f'$. Ergodicity and related properties of these classical operators have been studied on Banach spaces of entire functions, see e.g.\ \cite{Be14}, \cite{BeBoFe13}. On $H(\C)$ and $H(\D)$, via the topological isomorphism $\Phi$, they are conjugate to suitable weighted shift operators. 

In addition to these operators on spaces on holomorphic functions, we apply our results to the annihilation operator and creation operator on the space of rapidly decreasing smooth functions which we denote as usual by $\mathscr{S}(\R)$, i.e.\ on the space of complex valued smooth functions $f$ on $\R$ which have the property that $\sup_{x\in\R}(1+|x|)^k |f^{(j)}(x)|<\infty$ for all $k,j\in\N_0$. We equip $\mathscr{S}(\R)$ with its standard topology which makes it a (nuclear) Fr\'echet space. The \emph{creation operator} $A_+:\mathscr{S}(\R):\to\mathscr{S}(\R)$ and the \emph{annihilation operator} $A_-:\mathscr{S}(\R):\to\mathscr{S}(\R)$ are defined by
$$ A_+(f)=\frac{1}{\sqrt{2}}\left(-f'+xf\right), \quad A_-(f)=\frac{1}{\sqrt{2}}\left(f'+xf\right),$$
respectively, where we denote the multipication operator with the identity on $\mathscr{S}(\R)$ simply by $xf$. As is well known, via Hermite expansion in $L_2(\R)$, $\mathcal{S}(\R)$ is topologically isomorphic to the power series space $s=\Lambda_\infty\left((\ln(n+1))_{n\in\N_0}\right)$ and via this topological isomorphism, the annihilation and creation operators are conjugate to suitable weighted shift operators on $s$.

\begin{theo}\label{theo: examples}
	With the notation from above, the following hold.
	\begin{enumerate}
		\item[(i)] For $G\in\{\C,\D\}$ the operator $\Delta_0:H(G)\to H(G)$ is power bounded and (uniformly) mean ergodic.
		\item[(ii)] For $G\in\{\C,\D\}$ the Volterra operator $V:H(G)\to H(G)$ is power bounded and (uniformly) mean ergodic.
		\item[(iii)] For $G\in\{\C,\D\}$ the differentiation operator $\frac{d}{dz}:H(G)\to H(G)$ is topologizable but not (uniformly) mean ergodic.
		\item[(iv)] The annihilation operator $A_+:\mathscr{S}(\R)\to\mathscr{S}(\R)$ and the creation operator $A_-:\mathscr{S}(\R)\to\mathscr{S}(\R)$ are both topologizable but neither of them is power bounded nor (uniformly) mean ergodic.
	\end{enumerate}
\end{theo}
\begin{proof}
As already mentioned in the introduction to the theorem, the space $H(\C)$ of entire functions is topologically isomorphic to the infinite type power series space $\Lambda_\infty((n+1)_{n\in\N_0})$, i.e.\ $\alpha_n=n+1, n\in\N_0$, while the space $H(\D)$ is topologically isomorphic to the finite type power series space $\Lambda_0((n+1)_{n\in\N_0})$, i.e.\ $\alpha_n=n+1, n\in\N_0$. In both cases, the topological isomorphism is given by the mapping which takes a holomorphic function $f$ to the sequence of its Taylor coefficients in the origin. Via this mapping the continuous linear operator $\Delta_0$ is conjugate to the (unweighted) backward shift $B$ on $\Lambda_\infty((n+1)_{n\in\N_0})$, respectively on $\Lambda_0((n+1)_{n\in\N_0})$. Therefore, it follows immediately from Proposition \ref{prop: power bounded on power series spaces} (i), respectively (iii), that $\Delta_0$ is power bounded and therefore uniformly mean ergodic on $H(\C)$ and $H(\D)$, respectively, which proves (i).

(ii) Similarly to (i), the operator $V:H(\C)\to H(\C)$ is conjugate to the weighted forward shift $F_w$ with weight sequence $w=\left(\frac{1}{\max\{1,n\}}\right)_{n\in\N_0}$ on $\Lambda_\infty((n+1)_{n\in\N_0})$ and $V:H(\D)\to H(\D)$ is conjugate to the weighted forward shift $F_w$ with weight sequence $w=\left(\frac{1}{\max\{1,n\}}\right)_{n\in\N_0}$ on $\Lambda_0((n+1)_{n\in\N_0})$. Because
$$\limsup_{n+m\rightarrow\infty}\frac{\ln\left|n!/(n+m)!\right|}{n+m+1}\leq 0,$$
by Proposition \ref{prop: power bounded on power series spaces} (iv), $V$ is power bounded, thus, in particular, uniformly mean ergodic on $H(\D)$.

We shall prove that $V$ is power bounded on $H(\C)$ by applying Proposition \ref{prop: power bounded on power series spaces} (ii). Thus, let $k\in\N_0$ be arbitrary. We have to show that 
\begin{equation}\label{ex: operators on entire functions power bounded}
\frac{\ln\left|n!/(n+m)!\right|+k(n+m+1)}{n+1}=\frac{k(n+m+1)-\sum_{j=1}^m \ln(n+j)}{n+1}.
\end{equation}
has an upper bound independent of $n\in \N_0$ and $m\in \N$. In order to do so, we distinguish three cases depending on the choices of $n$ and $m$.

Case 1: Fix an arbitrary $N>0$. Consider every $n<N$ and let $m\to\infty$. Then using Stirling's formula we obtain
\begin{align*}
\lim_{m\to \infty} &\frac{k(n+m+1)-\sum_{j=1}^m \ln(n+j)}{n+1}\leq \lim_{m\to \infty}k+\frac{km}{n+1} - \frac{\sum_{j=1}^m \ln(j)}{n+1}= \lim_{m\to \infty}k+\frac{km}{n+1} - \frac{\ln(m!)}{n+1}\\
=& \lim_{m\to \infty} k+\frac{km}{n+1} - \frac{(m+\tfrac{1}{2})\ln(m)-m}{n+1}=\lim_{m\to \infty} k+\frac{(k-\ln(m)-1)m}{n+1} - \frac{\ln(m)}{2(n+1)}\leq k<\infty,
\end{align*}
for every $n<N$.

Case 2: Fix an arbitrary $M>0$. Consider every $m<M$ and let $n\to\infty$. Then we have
\begin{align*}
\lim_{n\to \infty} &\frac{k(n+m+1)-\sum_{j=1}^m \ln(n+j)}{n+1}\leq \lim_{n\to \infty} \frac{k(n+m+1)}{n}-\frac{m\ln(n)}{n}\\
\leq& \lim_{n\to \infty} k + \frac{(k-\ln(n))m}{n}\leq k<\infty.
\end{align*}

Case 3: We choose both $n$ and $m$ tending to $\infty$. To that aim we introduce the following notation. Given a matrix $\left(A_{n,m}\right)_{n\in\N_0,m\in\N}\in\R^{\N_0\times\N}$ we define the superior limit along $\min(n,m)$ as
\[ \limsup_{\min(n,m)\to \infty} A_{n,m}:= \inf_{k\in\N}\left\{ \sup_{\min(n,m)\geq k}\left\{ A_{n,m} \right\}\right\}. \]
Now we compute this superior limit for \eqref{ex: operators on entire functions power bounded} using Stirling's formula:
\begin{align*}
\limsup_{\min(n,m)\to \infty} & \frac{\ln\left|n!/(n+m)!\right|+k(n+m+1)}{n+1}\\ =  &\limsup_{\min(n,m)\to \infty} k+\frac{km}{n+1} + \frac{(n+\tfrac{1}{2})\ln(n)-n}{n+1}-\frac{(n+m+\tfrac{1}{2})\ln(n+m)-n-m}{n+1}\\ \leq 
&\limsup_{\min(n,m)\to \infty} k+\frac{(k-\ln(n+m)-1)m}{n+1}+ \frac{(n+\tfrac{1}{2})(\ln(n)-\ln(n+m))}{n+1}\leq k <\infty.
\end{align*}
Combining the above three cases we conclude that
$$\forall\,k\in\N_0:\,\sup_{n\in\N_0,m\in\N}\frac{\ln\left|n!/(n+m)!\right|+k(n+m+1)}{n+1}<\infty.$$ 
Thus, by Proposition \ref{prop: power bounded on power series spaces} (ii), $V$ is power bounded and therefore also uniformly mean ergodic on $H(\C)$ which proves statement (ii).

(iii) The operator $\frac{d}{dz}:H(\C)\to H(\C)$ is conjugate to the weighted backward shift $B_w$ with weight sequence $w=(n+1)_{n\in\N_0}$ on $\Lambda_\infty((n+1)_{n\in\N_0})$. While it is well known that $\frac{d}{dz}$ is topologically transitive and thus, it cannot be power bounded on $H(\C)$, for arbitrary $k,l\in\N_0$ it follows from
\begin{align*}
\sup_{r,n\in\N, r\geq n}\frac{1}{n} &\sum_{m=1}^n\frac{r!}{(r-m)!}e^{k(r-m)-lr}\geq\sup_{r,n\in\N,r\geq n}\frac{e^{r(k-l)}}{n}\sum_{m=1}^n(m-1)!e^{-km}\\
\geq&\sup_{r\in\N}\frac{e^{r(k-l)}}{r}(r-1)!e^{-kr}=\sup_{r\in\N}\frac{e^{-lr}r!}{r^2}
\geq\sup_{r\in\N}\frac{e^{-lr}\sqrt{2\pi r}e^{r\ln(r)-r}}{r^2}\\
=&\sqrt{2\pi}\sup_{r\in\N}\frac{e^{r(\ln(r)-(l+1))}}{r^{3/2}}=\infty
\end{align*}
together with Corollary \ref{cor: mean ergodicity on power series spaces} (i) that $\frac{d}{dz}$ is not mean ergodic on $H(\C)$. On the other hand, with Stirling's formula,
\begin{eqnarray*}
	&&\sup_{m\in\N}\limsup_{n\rightarrow\infty}\frac{\ln\left|\frac{(m+n+1)!}{n!}\right|}{n+m+1}\\
	&=&\sup_{m\in\N}\limsup_{n\rightarrow\infty}\frac{(m+n+\frac{3}{2})\ln(m+n+1)-(m+n+1)-(n+\frac{1}{2})\ln(n)+n}{n}\\
	&=&\sup_{m\in\N}\limsup_{n\rightarrow\infty}\frac{(m+1)\ln(m+n+1)+(n+\frac{1}{2})\ln\left(\frac{m+n+1}{n}\right)-(m+1)}{n}=0.
\end{eqnarray*}
Thus, by Proposition \ref{prop: topologizable on power series spaces} (i), $\frac{d}{dz}$ is topologizable on $H(\C)$. Of course, the latter can also be derived directly with the aid of the Cauchy formulas for derivatives.

Likewise, the differentiation operator $\frac{d}{dz}:H(\D)\to H(\D)$ is conjugate to the weighted backward shift $B_w$ with weight sequence $w=(n+1)_{n\in\N_0}$ on $\Lambda_0((n+1)_{n\in\N_0})$. Because for arbitrary $k,l\in\N_0$ it holds
\begin{align*}
\sup_{r,n\in\N, r\geq n}&\frac{1}{n}\sum_{m=1}^n\frac{r!}{(r-m)!}\exp\left(\frac{r}{l+1}-\frac{r-m}{k+1}\right)\geq\sup_{r,n\in\N, r\geq n}\frac{\exp\left(r\left(\frac{1}{l+1}-\frac{1}{k+1}\right)\right)}{n}\sum_{m=1}^n (m-1)!\\
\geq&\sup_{r\in\N}\frac{\exp\left(r\left(\frac{1}{l+1}-\frac{1}{k+1}\right)\right)r!}{r^2}
\geq\sqrt{2\pi}\sup_{r\in\N}\frac{\exp\left(r\left(\frac{1}{l+1}-\frac{1}{k+1}+\ln(r)-1\right)\right)}{r^{3/2}}=\infty,
\end{align*}
by Corollary \ref{cor: mean ergodicity on power series spaces} (iii) it follows that $\frac{d}{dz}$ is not mean ergodic on $H(\D)$. However, due to
\begin{eqnarray*}
	\sup_{m\in\N}\limsup_{n\rightarrow\infty}\frac{\ln\left|\frac{(n+m)!}{n!}\right|}{n+m}&=&\sup_{m\in\N}\limsup_{n\rightarrow\infty}\frac{(n+m+\frac{1}{2})\ln(n+m)-(n+m)-(n+\frac{1}{2})\ln(n)+n}{n+m}\\
	&=&\sup_{m\in\N}\limsup_{n\rightarrow\infty}\frac{m\ln(n+m)-m+(n+\frac{1}{2})\ln(\frac{n+m}{n})}{n+m}=0,
\end{eqnarray*}
it follows from Proposition \ref{prop: topologizable on power series spaces} (iii) that $\frac{d}{dz}$ is topologizable on $H(\D)$. Alternatively, as mentioned in before, this can be seen directly from Cauchy's formulas for derivatives.  

(iv) The mapping 
$$H:\mathscr{S}(\R)\rightarrow s=\Lambda_\infty\left((\ln(n+1))_{n\in\N_0}\right),\quad f\mapsto \left(\langle f, H_n\rangle\right)_{n\in\N_0}$$
is a topological isomorphism, where $H_n, n\in\N_0,$ denotes the $n$-th Hermite function and $\langle f, H_n\rangle$ the $L_2(\R)$-scalar product of $f$ and $H_n$, see \cite[Example 29.5.2]{MeVo1997}. Via this topological isomorphism the creation operator $A_-:\mathscr{S}(\R)\to\mathscr{S}(\R)$ is conjugate to the weighted backward shift $B_{(\sqrt{n+1})_{n\in\N_0}} :s\to s$ while the annihilation operator $A_+:\mathscr{S}(\R)\to\mathscr{S}(\R)$ is conjugate to the weighted forward shift $F_{(\sqrt{n})_{n\in\N_0}}:s\to s$.

Employing again Stirling's formula as well as $\lim\limits_{n\rightarrow\infty}(n+1/2)\ln(1+m/n)=m$ we see that
\begin{eqnarray*}
	\sup_{m\in\N_0}\limsup_{n\rightarrow\infty}\frac{\ln\sqrt{\frac{(n+m)!}{n!}}}{\ln(n+m+1)}&=&\frac{1}{2}\sup_{m\in\N_0}\limsup_{n\rightarrow\infty}\frac{\left(n+m+\frac{1}{2}\right)\ln(n+m)-m-\left(n+\frac{1}{2}\right)\ln(n)}{\ln(n+m+1)}\\&=&\frac{1}{2}\sup_{m\in\N_0}\limsup_{n\rightarrow\infty}\frac{\left(n+\frac{1}{2}\right)\ln\left(1+\frac{m}{n}\right)-m+m\ln(n+m)}{\ln(n+m+1)}\\
	&=&\frac{1}{2}\sup_{m\in\N}m=\infty.
\end{eqnarray*}
Hence, from Proposition \ref{prop: topologizable on power series spaces} (i) we derive that $A_-$ is not topologizable, a fortiori neither power bounded nor (uniformly) mean ergodic, on $\mathscr{S}(\R)$.

Finally, similar as above,
$$\sup_{m\in\N_0}\limsup_{n\rightarrow\infty}\frac{\ln\sqrt{\frac{(n+m)!}{n!}}}{\ln(n+1)}=\frac{1}{2}\sup_{m\in\N_0}\limsup_{n\rightarrow\infty}\frac{\left(n+\frac{1}{2}\right)\ln\left(1+\frac{m}{n}\right)-m+m\ln(n+m)}{\ln(n+1)}=\infty.$$	
Therefore, applying Proposition \ref{prop: topologizable on power series spaces} (ii) to $k=0$, we see that $A_+$ is not topologizable, hence neither power bounded nor mean ergodic on $\mathscr{S}(\R)$. 
\end{proof}

\textbf{Acknowledgements.} We are very grateful to José Bonet for valuable comments and suggestions about this work. The first author was partially supported by the project PID2020-119457GB-100 funded by MCIN/AEI/10.13039/501100011033 and by “ERDF A way of making Europe”, and the second author was partially supported by the project GV AICO/2021/170.

\bibliographystyle{abbrv}
\bibliography{bib_KS}

\begin{thebibliography}{10}

\bibitem{ABR2009}
A.~A. Albanese, J.~Bonet, and W.~J. Ricker.
\newblock Mean ergodic operators in {F}r\'{e}chet spaces.
\newblock {\em Ann. Acad. Sci. Fenn. Math.}, 34(2):401--436, 2009.

\bibitem{ABR2010}
A.~A. Albanese, J.~Bonet, and W.~J. Ricker.
\newblock On mean ergodic operators.
\newblock In {\em Vector measures, integration and related topics}, volume 201
  of {\em Oper. Theory Adv. Appl.}, pages 1--20. Birkh\"{a}user Verlag, Basel,
  2010.

\bibitem{AlBoRi13}
A.~A. Albanese, J.~Bonet, and W.~J. Ricker.
\newblock Convergence of arithmetic means of operators in {F}r\'{e}chet spaces.
\newblock {\em J. Math. Anal. Appl.}, 401(1):160--173, 2013.

\bibitem{AlBoRi14}
A.~A. Albanese, J.~Bonet, and W.~J. Ricker.
\newblock Characterizing {F}r\'{e}chet-{S}chwartz spaces via power bounded
  operators.
\newblock {\em Studia Math.}, 224(1):25--45, 2014.

\bibitem{AlJoMe22}
A.~A. Albanese, E.~Jord\'{a}, and C.~Mele.
\newblock Dynamics of composition operators on function spaces defined by local
  and global properties.
\newblock {\em J. Math. Anal. Appl.}, 514(1):Paper No. 126303, 15, 2022.

\bibitem{Be14}
M.~J. Beltr\'{a}n.
\newblock Dynamics of differentiation and integration operators on weighted
  spaces of entire functions.
\newblock {\em Studia Math.}, 221(1):35--60, 2014.

\bibitem{BeBoFe13}
M.~J. Beltr\'{a}n, J.~Bonet, and C.~Fern\'{a}ndez.
\newblock Classical operators on weighted {B}anach spaces of entire functions.
\newblock {\em Proc. Amer. Math. Soc.}, 141(12):4293--4303, 2013.

\bibitem{Beltran20}
M.~J. Beltr\'{a}n-Meneu.
\newblock Dynamics of weighted composition operators on weighted {B}anach
  spaces of entire functions.
\newblock {\em J. Math. Anal. Appl.}, 492(1):124422, 16, 2020.

\bibitem{BeGCJoJo16}
M.~J. Beltr\'{a}n-Meneu, M.~C. G\'{o}mez-Collado, E.~Jord\'{a}, and D.~Jornet.
\newblock Mean ergodic composition operators on {B}anach spaces of holomorphic
  functions.
\newblock {\em J. Funct. Anal.}, 270(12):4369--4385, 2016.

\bibitem{BGJJ2016mw}
M.~J. Beltr\'{a}n-Meneu, M.~C. G\'{o}mez-Collado, E.~Jord\'{a}, and D.~Jornet.
\newblock Mean ergodicity of weighted composition operators on spaces of
  holomorphic functions.
\newblock {\em J. Math. Anal. Appl.}, 444(2):1640--1651, 2016.

\bibitem{BeJo21}
M.~J. Beltr\'{a}n-Meneu and E.~Jord\'{a}.
\newblock Dynamics of weighted composition operators on spaces of entire
  functions of exponential and infraexponential type.
\newblock {\em Mediterr. J. Math.}, 18(5):Paper No. 212, 18, 2021.

\bibitem{BoPaRi11}
J.~Bonet, B.~de~Pagter, and W.~J. Ricker.
\newblock Mean ergodic operators and reflexive {F}r\'{e}chet lattices.
\newblock {\em Proc. Roy. Soc. Edinburgh Sect. A}, 141(5):897--920, 2011.

\bibitem{BoDo2011A}
J.~Bonet and P.~Doma\'{n}ski.
\newblock A note on mean ergodic composition operators on spaces of holomorphic
  functions.
\newblock {\em Rev. R. Acad. Cienc. Exactas F\'{\i}s. Nat. Ser. A Mat. RACSAM},
  105(2):389--396, 2011.

\bibitem{BoDo11B}
J.~Bonet and P.~Doma\'{n}ski.
\newblock Power bounded composition operators on spaces of analytic functions.
\newblock {\em Collect. Math.}, 62(1):69--83, 2011.

\bibitem{BoJoRo18}
J.~Bonet, E.~Jord\'{a}, and A.~Rodr\'{\i}guez.
\newblock Mean ergodic multiplication operators on weighted spaces of
  continuous functions.
\newblock {\em Mediterr. J. Math.}, 15(3):Paper No. 108, 11, 2018.

\bibitem{BoRi09}
J.~Bonet and W.~J. Ricker.
\newblock Mean ergodicity of multiplication operators in weighted spaces of
  holomorphic functions.
\newblock {\em Arch. Math. (Basel)}, 92(5):428--437, 2009.

\bibitem{DuSc1958}
N.~Dunford and J.~T. Schwartz.
\newblock {\em Linear {O}perators. {I}. {G}eneral {T}heory}.
\newblock Pure and Applied Mathematics, Vol. 7. Interscience Publishers, Inc.,
  New York; Interscience Publishers Ltd., London, 1958.
\newblock With the assistance of W. G. Bade and R. G. Bartle.

\bibitem{EiFaHaNa15}
T.~Eisner, B.~Farkas, M.~Haase, and R.~Nagel.
\newblock {\em Operator theoretic aspects of ergodic theory}, volume 272 of
  {\em Graduate Texts in Mathematics}.
\newblock Springer, Cham, 2015.

\bibitem{FeGaJo18}
C.~Fern\'{a}ndez, A.~Galbis, and E.~Jord\'{a}.
\newblock Dynamics and spectra of composition operators on the {S}chwartz
  space.
\newblock {\em J. Funct. Anal.}, 274(12):3503--3530, 2018.

\bibitem{GCJoJo16}
M.~C. G\'{o}mez-Collado, E.~Jord\'{a}, and D.~Jornet.
\newblock Power bounded composition operators on spaces of meromorphic
  functions.
\newblock {\em Topology Appl.}, 203:141--146, 2016.

\bibitem{HaShZh19}
S.-A. Han and Z.-H. Zhou.
\newblock Mean ergodicity of composition operators on {H}ardy space.
\newblock {\em Proc. Indian Acad. Sci. Math. Sci.}, 129(4):Paper No. 45, 10,
  2019.

\bibitem{Jarchow1981}
H.~Jarchow.
\newblock {\em Locally convex spaces}.
\newblock B. G. Teubner, Stuttgart, 1981.
\newblock Mathematische Leitf\"{a}den. [Mathematical Textbooks].

\bibitem{JoSaSP20}
D.~Jornet, D.~Santacreu, and P.~Sevilla-Peris.
\newblock Mean ergodic composition operators in spaces of homogeneous
  polynomials.
\newblock {\em J. Math. Anal. Appl.}, 483(1):123582, 11, 2020.

\bibitem{JoSaSP21}
D.~Jornet, D.~Santacreu, and P.~Sevilla-Peris.
\newblock Mean ergodic composition operators on spaces of holomorphic functions
  on a {B}anach space.
\newblock {\em J. Math. Anal. Appl.}, 500(2):125139, 16, 2021.

\bibitem{K2019p}
T.~Kalmes.
\newblock Power bounded weighted composition operators on function spaces
  defined by local properties.
\newblock {\em J. Math. Anal. Appl.}, 471(1-2):211--238, 2019.

\bibitem{K2020}
T.~Kalmes.
\newblock Topologizable and power bounded weighted composition operators on
  spaces of distributions.
\newblock {\em Ann. Polon. Math.}, 125(2):139--154, 2020.

\bibitem{KaSa22}
T.~Kalmes and D.~Santacreu.
\newblock Mean ergodic composition operators on spaces of smooth functions and
  distributions.
\newblock {\em Proc. Amer. Math. Soc.}, 150(6):2603--2616, 2022.

\bibitem{Krengel}
U.~Krengel.
\newblock {\em Ergodic theorems}, volume~6 of {\em De Gruyter Studies in
  Mathematics}.
\newblock Walter de Gruyter \& Co., Berlin, 1985.
\newblock With a supplement by Antoine Brunel.

\bibitem{MeVo1997}
R.~Meise and D.~Vogt.
\newblock {\em Introduction to functional analysis}, volume~2 of {\em Oxford
  Graduate Texts in Mathematics}.
\newblock The Clarendon Press Oxford University Press, New York, 1997.
\newblock Translated from the German by M. S. Ramanujan and revised by the
  authors.

\bibitem{Piszczek10b}
K.~Piszczek.
\newblock Barrelled spaces and mean ergodicity.
\newblock {\em Rev. R. Acad. Cienc. Exactas F\'{\i}s. Nat. Ser. A Mat. RACSAM},
  104(1):5--11, 2010.

\bibitem{Piszczek10}
K.~Piszczek.
\newblock Quasi-reflexive {F}r\'{e}chet spaces and mean ergodicity.
\newblock {\em J. Math. Anal. Appl.}, 361(1):224--233, 2010.

\bibitem{Rodriguez19}
A.~Rodr\'{\i}guez-Arenas.
\newblock Some results about diagonal operators on {K}\"{o}the echelon spaces.
\newblock {\em Rev. R. Acad. Cienc. Exactas F\'{\i}s. Nat. Ser. A Mat. RACSAM},
  113(4):2959--2968, 2019.

\bibitem{SeMeBo20}
W.~Seyoum, T.~Mengestie, and J.~Bonet.
\newblock Mean ergodic composition operators on generalized {F}ock spaces.
\newblock {\em Rev. R. Acad. Cienc. Exactas F\'{\i}s. Nat. Ser. A Mat. RACSAM},
  114(1):Paper No. 6, 11, 2020.

\end{thebibliography}

\end{document}